\documentclass{article}
\usepackage{graphicx} % Required for inserting images
\usepackage{amsmath,amstext,amssymb,amsopn,amsthm,color}
\usepackage[margin=1.0in]{geometry}
\usepackage{enumitem}
\usepackage{hyperref}
\usepackage{soul}
\numberwithin{equation}{section}

\newcommand{\R}{\mathbb{R}}
\newcommand{\N}{\mathbb{N}}
\newcommand{\pr}{\mathbb{P}}
\newcommand{\cond}[1]{\, #1 \vert \,}

\newcommand{\inner}[2]{\left<#1,#2\right>}

\newcommand{\AN}{A_{N-1}}

\newcommand{\C}{W}
\newcommand{\Cc}{\overline{W}}
\newcommand{\Cb}{\partial \C}

\newcommand{\Cbm}{\partial \C^{(m)}}

\newcommand{\Aaclear}{\mathcal{A}}
\newcommand{\Aa}{\mathcal{A}^{w}}

\newcommand{\xa}{\left<x,\alpha\right>}
\newcommand{\xb}{\left<x,\beta\right>}
\newcommand{\xc}{\left<x,\gamma\right>}

\allowdisplaybreaks
\newcommand{\xta}{\left<x(t),\alpha\right>}
\newcommand{\xtb}{\left<x(t),\beta\right>}

\newcommand{\xas}{\left<x,\alpha^*\right>}

\newcommand{\xtas}{\left<x(t),\alpha^*\right>}
\newcommand{\xbs}{\left<x,\beta^*\right>}
\newcommand{\xtbs}{\left<x(t),\beta^*\right>}

\newcommand{\xcs}{\left<x,\gamma^*\right>}

\newcommand{\xtcs}{\left<x(t),\gamma^*\right>}
\newcommand{\xds}{\left<x,\delta^*\right>}

\newcommand{\aR}{\alpha\in R_+}
\newcommand{\bR}{\beta\in R_+}

\newcommand{\SaR}{\sum_{\alpha\in R_+}}

\newcommand{\SabR}{\sum_{\alpha\in R_+}\sum_{\beta\in R_+}}
\newcommand{\SanbR}{\mathop{\sum_{\aR}\sum_{\bR}}_{\alpha\neq \beta}}

\newcommand{\SaPz}{\sum_{\alpha\in\Pzero}}

\newcommand{\SbPzna}{\sum_{\beta\in\Pzero\backslash\{\alpha\}}}
\newcommand{\SabPz}{\sum_{\alpha,\beta\in\Pzero}}
\newcommand{\SaRnb}{\sum_{\aR\backslash\{\beta\}}}

\newcommand{\Sni}{\sum_{i=1}^N}

\newcommand{\ab}{\left<\alpha,\beta\right>}
\newcommand{\ac}{\left<\alpha,\gamma\right>}
\newcommand{\cb}{\left<\gamma,\beta\right>}
\newcommand{\bc}{\left<\beta,\gamma\right>}

\newcommand{\abs}{\left<\alpha^*,\beta^*\right>}

\newcommand{\bcs}{\left<\beta^*,\gamma^*\right>}

\newcommand{\yy}{\left<y,y\right>}
\newcommand{\yz}{\left<y,z\right>}
\newcommand{\reflect}[1]{\varrho_{#1}}

\newcommand{\shift}{\,\,}

\newcommand{\eA}[1]{e_{#1}(\Aa)}

\newcommand{\eaA}[1]{e_{#1}^{\shift\overline{\alpha}}(\Aa)}
\newcommand{\ebA}[1]{e_{#1}^{\shift\overline{\beta}}(\Aa)}

\newcommand{\eabA}[1]{e_{#1}^{\overline{\alpha}, \overline{\beta}}(\Aa)}

\newcommand{\eabcA}[1]{e_{#1}^{\overline{\alpha}, \overline{\beta}, \overline{\gamma}}(\Aa)}

\newcommand{\eAc}[1]{e_{#1}(\Aaclear)}
\newcommand{\eaAc}[1]{e_{#1}^{\shift\overline{\alpha}}(\Aaclear)}

\newcommand{\eabAc}[1]{e_{#1}^{\overline{\alpha}, \overline{\beta}}(\Aaclear)}
\newcommand{\eacAc}[1]{e_{#1}^{\overline{\alpha}, \overline{\gamma}}(\Aaclear)}

\newcommand{\Pzero}{\mathcal{P}^0(y)}
\newcommand{\Pplus}{\mathcal{P}^+(y)}
\newcommand{\domain}{D}

\newcommand{\Pzerox}{\mathcal{P}^0(x)}

\newcommand{\SimplePlus}{\Delta_+}

\newcommand{\lifetime}{T_\infty}

\theoremstyle{plain}
\newtheorem{theorem}{Theorem}
\newtheorem{lemma}[theorem]{Lemma}
\newtheorem{corollary}[theorem]{Corollary}
\newtheorem{proposition}{Proposition}

\newtheorem{definition}{Definition}

\newtheorem{remark}{Remark}

\usepackage{xcolor}

\usepackage{comment}
\usepackage{cancel}

\title{Collision types and times in interacting particle systems}
\author{Sergio Andraus$^1$, Nicole Hufnagel$^2$, Jacek Małecki$^3$}
\date{
	$^1$Japan International University\\%
	$^2$Heinrich Heine University Dusseldorf\\
	$^3$Wrocław University of Science and Technology\\[2ex]%
	\today}

\begin{document}

\maketitle

\begin{abstract}
    We consider a system of stochastic interacting particles with general diffusion coefficient and drift functions and we study the types of collisions that arise in them.
    In particular, interactions between particles are inversely proportional to their separation, and the coupling function of interaction is also considered in great generality.    Our main result shows that, under positivity and the stated variance compatibility conditions, no two distinct positive-root hyperplanes are reached simultaneously at a positive time. In type $A_{N-1}$, this excludes both collisions involving three or more particles and simultaneous collisions of disjoint pairs.
    In order to obtain our results we make use of symmetric polynomials in squared root projections;
    the degree of these polynomials indicates the type of collision, and by a locality argument we show that polynomials indicating a non-simple collision almost surely do not cancel. We use this result to obtain upper and lower bounds for the Hausdorff dimension of the set of collision times in terms of interaction-to-variance ratios. These bounds coincide in particular constant-ratio cases.
    Our results cover many of the most well-known particle systems, such as the Dyson model and Wishart processes and their extensions to non-constant diffusion coefficients and background drifts.
\end{abstract}

\section{Introduction}

Stochastic particle systems have been an important topic of study in both mathematics and physics. In particular, systems like the Dyson model in nuclear physics \cite{Dyson62A} and Wishart processes in statistics \cite{Bru91} are representative due to their connections to random matrix theory and integrable systems \cite{KatoriTanemura04, Forrester10}. These particle systems are special cases of the more general particle system $x = (x_1,\ldots,x_N)$ we consider in this paper, which is defined as the solution to the following SDE
\begin{equation}
	\label{eq:x:SDE}
	dx_i(t) = \sigma\big(x_i(t)\big)dB_i(t)+b\big(x_i(t)\big)dt+\SaR k_{\alpha}\big(x(t)\big)\dfrac{\alpha_i}{\xta} \,dt.
\end{equation}
Here, $R_+$ denotes a positive subsystem of a fixed root system $R$, $\C$ is its open Weyl chamber, and $x$ takes values in $\Cc$. We use the Euclidean inner product $\langle\cdot,\cdot\rangle$. The roots and chambers are described explicitly in Section~\ref{sec:roots}. The process is considered up to its explosion time $T_\infty=\lim_{L\to\infty}T_L$, with $T_L=\inf\{t\geq0:|x(t)|\geq L\}$ and $\inf\varnothing=\infty$. All statements concern times strictly before $T_\infty$.

We can also write the above system of equations in the form of a single multivariate equation
\begin{equation}
	\label{eq:x:SDE:mult}
	dx = \sigma(x)dB+b(x)dt+\SaR k_{\alpha}(x)\dfrac{\alpha}{\xa} \,dt\/,
\end{equation}
where we slightly disrupt notation by writing $\mathbf \sigma(x)=(\sigma(x_1),\ldots,\sigma(x_N))$ and $\mathbf b(x) = (b(x_1),\ldots,b(x_N))$. We will use this notation in the cases where we talk about the diffusion coefficient $\sigma$ and the drift function $b$ as functions of $N$ variables.
In contrast, the coupling interaction function $k_\alpha(x)$ always indicates a multivariate function. The components of $B$ are independent standard Brownian motions. The product $\sigma(x)dB$ in \eqref{eq:x:SDE:mult} is understood componentwise: its $i$th component is $\sigma(x_i)dB_i$.

We see from the form of \eqref{eq:x:SDE} that there is a singularity when $\xa$ approaches zero, which is one of the most interesting features of the class of systems we consider. We call a hit of $\langle x,\alpha\rangle=0$ a root-wall collision. For $\alpha=e_j-e_i$ this is the ordinary particle collision $x_i=x_j$; for $\alpha=e_j+e_i$ it is the relation $x_j=-x_i$; and for $\alpha=e_i$ it is a hit of the origin. The latter two interpretations are relevant in types $B_N$ and $D_N$, and in type $B_N$, respectively. Thus, a root-wall collision need not be an ordinary pair collision.
Moreover, it is known that for a subset of the processes we consider, the first time in which $x$ hits $\Cb$ is almost surely finite when the coupling interaction function is sufficiently weak \cite{Demni09}.
The first boundary hitting time and the lifetime are different: the solution is not stopped at a collision, and recurrent simple collisions may occur before its lifetime.

Allowing nonconstant diffusion and configuration-dependent coupling makes it possible to describe fluctuations whose intensity varies with position and interactions whose strength responds to the particle configuration. This raises a natural robustness question: which features of collision behaviour in the classical models persist under these variations? In particular, our first objective is to exclude simultaneous hits of distinct root hyperplanes. In type $A_{N-1}$, this includes both a collision involving at least three particles and simultaneous collisions of disjoint pairs. Our second objective is to bound the Hausdorff dimension of the full set of collision times providing detailed conditions for their almost-sure occurrence.

The construction and analysis of systems of the form \eqref{eq:x:SDE} rest on a rich literature. The prototype of such systems is Dyson's Brownian motion model for matrix eigenvalues \cite{Dyson62A}, and later on Cépa and Lépingle introduced the framework of multivalued stochastic differential equations for particles with electrostatic repulsion \cite{CepaLepingle} and proved, in particular, that particles governed by a sufficiently strong logarithmic repulsion never collide; in a subsequent note they showed that, even when repulsion is too weak to prevent collisions, multiple collisions at a positive time do not occur \cite{CepaLepingle07}. In the root-system setting, Demni established existence, uniqueness and hitting-time results for radial Dunkl processes \cite{Demni09}, which constitute the special case of \eqref{eq:x:SDE} with constant Weyl-invariant multiplicities. Strong solutions of non-colliding systems were obtained by Graczyk and Małecki, first via a multidimensional Yamada–Watanabe theorem \cite{GraczykMalecki13} and then for general ordered particle systems with repulsive interactions \cite{GraczykMalecki}; the same authors later studied squared Bessel particle systems, in which the repulsion may be degenerate and collisions are in fact allowed \cite{GraczykMalecki19}. Non-colliding diffusion limits and examples in this spirit include the Wishart and Jacobi processes \cite{Bru91,Demni10}, as well as the non-colliding systems of Katori and Tanemura \cite{KatoriTanemura04} and the eigenvalue processes studied by König and O'Connell \cite{KonigOConnell01}. The existence of solutions to \eqref{eq:x:SDE} in the full generality needed for the present work (with configuration-dependent diffusion, drift and repulsion coefficients for the root systems $A_{N-1}$, $B_N$, and $D_N$) has recently been established by Małecki in \cite{Malecki25}, where a general theory of generalized Bessel–Dunkl diffusions is developed: weak existence is proved there through a desingularization by invariant polynomials. Non-explosion, pathwise uniqueness, strong existence, and mean-field limits are obtained there under their respective additional hypotheses. The existence proof in \cite[Sections~5-7]{Malecki25} establishes the collision properties required for its construction within that paper, without invoking the present theorem. Its Theorem~2 already excludes simultaneous hits of distinct root walls whose roots share a coordinate, for the solution constructed there. Our Theorem~\ref{th:multiplecollisions} additionally excludes simultaneous hits of walls with disjoint coordinate supports and applies to every solution satisfying our assumptions. We also determine conditions for the existence of collisions and bounds for the Hausdorff dimension of the set of collision times.
It must be noted that the system in \cite{Malecki25} contains \eqref{eq:x:SDE} in the sense that all coefficients in their SDE (diffusion, drift, and interaction strength) are time-dependent, whereas our system has coefficients that depend only on the current system configuration.

Before we get into more details regarding our results, we briefly mention that we study \eqref{eq:x:SDE} not only for its generality but also for a clear physical motivation. Interacting particle systems with singular interactions can be interpreted as charged particles in one dimension, and are most often studied in quite rigid conditions: temperature is held constant in space and time, only very simple background potentials, if any, are considered, and the interaction strength is a constant. The system described by \eqref{eq:x:SDE} addresses these three issues, as it allows for fixed temperature gradients in the form of non-constant $\sigma$, a general background potential that appears as the drift function $b$, and localized charge screening due to the space dependence of the interaction strength $k_\alpha$. This allows us to show the extent to which the results found in simpler systems still apply, and under which conditions they break down. In this paper we mainly look at particle collisions, and how their frequency, expressed in terms of Hausdorff dimensions, changes depending on the properties of $\sigma$, $b$ and $k_\alpha$.

Let us introduce the fundamental conditions required for the stating our results. Since the form of the considered SDEs \eqref{eq:x:SDE} is very general, we introduce the following collection of assumptions, which establish the general framework for our considerations. We begin with very general assumptions. The functions $\sigma,b$ in \eqref{eq:x:SDE} are single-valued and defined on the domain 
    \begin{equation*}
        D=\{u\in\mathbb R:\text{there are }z\in\overline W \text{ and }i\text{ with }z_i=u\}.
    \end{equation*}
    
\begin{enumerate}[label=(G\arabic*), leftmargin=2cm]
    \item \label{ass:sigma:b:continuous} 
    The functions $\sigma, b:\domain \to \R$ and $k_\alpha:\Cc\to \R$ are continuous for every $\alpha\in R$. 

    \item \label{ass:R:root system} 
    We restrict the root system to \begin{align*}
        R\in \begin{cases}
            A_{N-1}&:=\{\pm(e_i-e_j) \cond{} 1\leq j< i\leq N\},\\
            B_N&:=\{\pm e_i\pm e_j \cond{} 1\leq j< i\leq N\}\cup \{\pm e_i \cond{}1\leq i \leq N \},\\
            D_N&:= \{\pm e_i\pm e_j \cond{}1 \leq i< j\leq N\}
        \end{cases}
    \end{align*}
    The signs in $\pm e_i\pm e_j$ are independent. For distinct, non-opposite roots $\alpha$ and $\beta$ with nonzero inner product, $|\langle\alpha,\beta\rangle|=1$. 
    Moreover, the SDEs \eqref{eq:x:SDE} and \eqref{eq:x:SDE:mult} are identical for $B_N$ and $C_N:=\{\pm e_i\pm e_j \cond{} 1\leq j< i\leq N\}\cup \{\pm 2e_i \cond{}1\leq i \leq N \}$, so the latter case is covered by the former. Note that these root systems are both reduced and crystallographic.
    \item \label{ass:variance:compatibility}
    In type $D_N$, assume that $\sigma^2(u)=\sigma^2(-u)$ for every $u\in\mathbb R$. In types $A_{N-1}$ and $B_N$, no additional condition is imposed here.
\end{enumerate}
    The restriction to these root systems is used in the reflection identities and covariance bounds underlying the collision proof. Their explicit positive roots and Weyl chambers are described in Section~\ref{sec:roots}. An extension to other root systems is outside the scope of the present paper.

    To present the last assumptions let us recall that every positive root $\alpha\in R_+$ has a unique expansion
        \begin{equation*}
            \alpha=\sum_{\delta\in\SimplePlus}n_\delta(\alpha)\delta, \qquad n_\delta(\alpha)\in\mathbb N_{0}.
        \end{equation*}
    We use the root order
        \begin{equation*}
            \alpha\leq\gamma \quad\Longleftrightarrow\quad n_\delta(\alpha)\leq n_\delta(\gamma) \quad\text{for every }\delta\in\SimplePlus.
        \end{equation*} 
    In particular, $\alpha\leq\gamma$ implies $\langle x,\alpha\rangle\leq\langle x,\gamma\rangle$ for every $x\in\Cc$.
    This allows us to complement the list of general assumptions with the following. 
    \begin{enumerate}[label=(A\arabic*), leftmargin=2cm]
        \item \label{ass:sigma:k:positive} 
        The functions $\sigma$ and $k_\alpha$ are strictly positive for every $\alpha\in R_+$.
        \item \label{ass:b:monotone} 
        The function $b$ satisfies the inequality 
            \begin{equation*}
                \sum_{i=1}^N \alpha_i b(x_i)\leq0.
            \end{equation*}
        for every $x\in\Cc$ and$\alpha\in\SimplePlus$. 

    \item \label{ass:k:monotonicty}
        For every $\alpha,\gamma\in R_+$ such that $\alpha\leq\gamma$ and $|\alpha|=|\gamma|$, we have
            \begin{equation*}
                k_\gamma(x)\langle x,\alpha\rangle \leq k_\alpha(x)\langle x,\gamma\rangle, \qquad x\in\Cc.
            \end{equation*}    
\end{enumerate}
    In type $A_{N-1}$, condition \ref{ass:b:monotone} means that $b$ is non-increasing. In type $B_N$, for $N\geq2$, it means that $b$ is non-increasing and non-positive on $[0,\infty)$; for $B_1$, only non-positivity is required. In type $D_N$, its explicit form is
        \begin{equation*}
            b(x_{i+1})-b(x_i)\leq0,\quad 1\leq i<N, \qquad b(x_1)+b(x_2)\leq0, \qquad x\in\Cc.
        \end{equation*}
    For example, $b(u)=-\lambda u$, $\lambda\geq0$, satisfies this condition in all three types.

        The condition \ref{ass:k:monotonicty} expresses monotonicity of the repulsive force with respect to the root order, within each root-length class. In type $A_{N-1}$ it gives the usual comparison of forces associated with nested particle intervals. In type $B_N$, the restriction to equal-length roots permits independent short-root and long-root couplings.
    
    Under \ref{ass:sigma:b:continuous}--\ref{ass:variance:compatibility} and \ref{ass:sigma:k:positive}, an application of \cite[Theorem~1]{Malecki25}, with localization, ensures the existence of a weak solution to \eqref{eq:x:SDE:mult}, started from any $x_0\in\Cc$, up to its explosion time $T_\infty$. The variance compatibility required there holds at difference walls because the colliding coordinates are equal. At sum walls, it follows from nonnegativity of the coordinates in type $B_N$ and from \ref{ass:variance:compatibility} in type $D_N$. Continuity then gives the required local compatibility near these walls.

    Throughout the paper, $x$ denotes any such solution. It is a continuous $\Cc$-valued semimartingale satisfying \eqref{eq:x:SDE:mult} in its integral form; in particular, almost surely, for every $\alpha\in R_+$ and every $t<T_\infty$,
        \begin{equation*}
            \int_0^t \frac{k_\alpha(x(s))}{\langle x(s),\alpha\rangle}\,ds <\infty.
        \end{equation*}

The first result, which enables us to explore more sophisticated properties of the set of collision times, describes the nature of collisions. 

\begin{theorem}
    \label{th:multiplecollisions}
    Assume that \ref{ass:sigma:b:continuous}-\ref{ass:variance:compatibility} and \ref{ass:sigma:k:positive} hold and let $x=(x_1,\ldots,x_N)$ be a solution to \eqref{eq:x:SDE:mult}. Then, there are no double collisions after the start, i.e. for every two different $\alpha,\beta \in R_+$ we have
	   \begin{equation*}
	        \xta^2+\xtb^2>0\/,\quad \textrm{for all $t\in(0,\lifetime)$}
	   \end{equation*}
    almost surely.
\end{theorem}

We note that this result is in contrast with those in \cite{Yuan24}, where matrix valued processes with Gaussian entries are considered, but the dynamics considered there are fractional.
Our results are applicable to matrix processes with Brownian motions as entries (see Section~\ref{sec:examples}), where, in fact, no collisions occur.
However, when matrix entries are fractional Brownian motions, there exist conditions that ensure the existence of multiple collisions with non-zero probability.

Having excluded simultaneous collisions, we turn to the size
and structure of the set of collision times
\begin{equation*}
    \mathcal T:=\{t\in(0,T_\infty):x(t)\in\partial W\}.
\end{equation*}
Knowing whether this set is nonempty does not describe how
collisions are distributed in time. Moreover, its Lebesgue
measure is zero under our assumptions, by the local
integrability of the singular drifts and the positivity of the interaction coefficients. A finer notion of size is therefore needed.

Hausdorff dimension extends the usual notion of dimension to sets whose dimension need not be an integer. For subsets of the time axis, it takes values between zero and one: countable sets have dimension zero, while nondegenerate intervals have dimension one. Sets of zero Lebesgue measure
may nevertheless have a strictly positive Hausdorff dimension. Informally, this dimension measures how efficiently a set can be covered by small intervals. One considers coverings by intervals $I_j$ and the sums
\begin{equation*}
    \sum_j |I_j|^s,
\end{equation*}
where $|I_j|$ denotes the length of $I_j$. The dimension is the critical exponent separating those $s$ for which these sums can be made arbitrarily small, using arbitrarily short covering intervals, from those for which they cannot. The precise definition is recalled in Section~\ref{sec:prelim}.

Classical examples arise from the zero sets of stochastic processes. If $B$ is a one-dimensional Brownian motion started at zero, then, almost surely, for every $T>0$,
\begin{equation*}
    \dim\{t\in[0,T]:B(t)=0\}=\frac12.
\end{equation*}
Thus, Brownian motion returns to zero at an uncountable set of times, although it spends zero Lebesgue time there. More generally, let $Z$ be a squared Bessel process with dimension parameter $\delta>0$, defined by
\begin{equation*}
    dZ(t)=2\sqrt{Z(t)}\,dB(t)+\delta\,dt,
    \qquad Z(0)=0.
\end{equation*}
Its square root is a Bessel process and has the same zero set. For $0<\delta<2$, the classical zero-set dimension formula gives, almost surely, for every $T>0$,
\begin{equation*}
    \dim\{t\in[0,T]:Z(t)=0\}=1-\frac{\delta}{2};
\end{equation*}
see \cite[Theorem 4.2]{LiuXiao98}. For $\delta\geq2$, the process is strictly positive at every positive time almost surely; see \cite[Corollary 1]{GoeingJaeshckeYor03}. The case $\delta=1$ recovers the Brownian zero-set dimension, since a one-dimensional Bessel process has the law of the absolute value of Brownian motion.

These results also motivate our approach to particle
collisions. Near a simple collision, we study the squared projection onto the normal to the corresponding root wall. After a change of time, this projection can be compared with squared Bessel processes whose dimension parameters are controlled by the interaction-to-variance ratios
\begin{equation*}
    \frac{|\alpha|^2k_\alpha(x)}
         {\sum_{i=1}^N\alpha_i^2\sigma^2(x_i)},
    \qquad \alpha\in\SimplePlus.
\end{equation*}
A squared Bessel dimension parameter of the form $\delta=1+2\eta$ gives the zero-set dimension
$1/2-\eta$ when $\eta<1/2$, explaining the form of the bounds below. Hausdorff dimension therefore quantifies the fractal size of the collision-time set. It does not represent a number of collisions per unit time, and a zero-dimensional collision set need not be empty.

The remaining results, which are consequences of Theorem~\ref{th:multiplecollisions}, are given in the following two statements.
\begin{theorem}
    Assume that \ref{ass:sigma:b:continuous}-\ref{ass:variance:compatibility} and \ref{ass:sigma:k:positive} hold, let $x=(x_1,\dots,x_N)$ be a solution to \eqref{eq:x:SDE:mult} and suppose that the ratios $b/\sigma^2$ and $k/\sigma^2$ are bounded everywhere in $\domain$ and $\Cc$. Then, the Hausdorff dimension of $x^{-1}(\partial W)$ is almost surely bounded above by
    \[\dim x^{-1}(\partial W)\leq \max\bigg\{0,\frac12-\min_{ {\alpha\in\SimplePlus}}\inf_{y\in W}\frac{|\alpha|^2k_\alpha(y)}{\Sni\alpha_i^2\sigma^2(y_i)}\bigg\}.\]
    \label{th:Hausdorffupper}
\end{theorem}
\begin{theorem}
    Assume that \ref{ass:sigma:b:continuous}-\ref{ass:variance:compatibility} and \ref{ass:sigma:k:positive}-\ref{ass:k:monotonicty} hold, let $x=(x_1,\dots,x_N)$ be a solution to \eqref{eq:x:SDE:mult}, suppose that the ratios $b/\sigma^2$ and $k/\sigma^2$ are bounded everywhere in $\domain$ and that, for some $a>0$ and $p\geq0$,
    \[
        \sigma(u)\geq a(1+|u|)^{-p}
    \]
    throughout its domain.
    Suppose also that there exists a deterministic $r\geq0$
with $2pr\leq1$ and an almost surely finite random variable
$K$ such that, almost surely,
    \begin{equation*}
        |x_i(t)|\leq K(1+t)^r, \qquad 0\leq t<T_\infty,\quad i=1,\ldots,N.
    \end{equation*}
    Then, the Hausdorff dimension of $x^{-1}(\partial W)$ is almost surely bounded below by
        \[\dim x^{-1}(\partial W)\geq \max \bigg\{0,\frac12 - \min_{\alpha\in \Delta_+}\sup_{y\in W}\frac{|\alpha|^2k_\alpha(y)}{\Sni\alpha_i^2\sigma^2(y_i)}\bigg\}.\]
    \label{th:Hausdorfflower}
\end{theorem}
\begin{remark}
Theorem~\ref{th:Hausdorfflower} indicates sufficient conditions on $k$, $b$, and $\sigma$ for almost-sure existence of collisions. In this statement, we have imposed a decay condition on $\sigma$ in order to deal with the situation where $|x(t)|\to\infty$ as $t\uparrow T_\infty$. This decay condition effectively dictates the growth speed of $|x(t)|$ at which the statement holds, and the slower the decay of $\sigma$, the faster $|x(t)|$ is allowed to grow. In fact, if $\sigma$ is bounded below by some positive constant, $r$ can be arbitrarily large. As will be shown in the proof, the case where $x(t)$ lies in a bounded region of $\Cc$ does not require the decay condition, and positivity of $\sigma$ is sufficient to prove the assertion. 
\end{remark}
A weaker, positive probability form of this last statement can be proved without the assumptions \ref{ass:b:monotone}, \ref{ass:k:monotonicty}, or the decay condition on $\sigma$.
\begin{lemma}
    Assume that \ref{ass:sigma:b:continuous}-\ref{ass:variance:compatibility} and \ref{ass:sigma:k:positive} hold, let $x=(x_1,\dots,x_N)$ be a solution to \eqref{eq:x:SDE:mult} and suppose that the ratios $b/\sigma^2$ and $k/\sigma^2$ are bounded everywhere in $\domain$ and $\Cc$. Then, the Hausdorff dimension of $x^{-1}(\partial W)$ is bounded below by
        \[\dim x^{-1}(\partial W)\geq \max \bigg\{0,\frac12 - \max_{\alpha\in \Delta_+}\sup_{y\in W}\frac{|\alpha|^2k_\alpha(y)}{\Sni\alpha_i^2\sigma^2(y_i)}\bigg\}\]
        with positive probability. Otherwise we have the trivial bound $\dim x^{-1}(\partial W)\geq 0$.
    \label{th:Hausdorfflowerweak}
\end{lemma}
Our dimension bounds describe collision behaviour for systems associated with the root systems $A_{N-1}$, $B_N$, and $D_N$, allowing nonconstant diffusion coefficients, background drifts, and configuration-dependent interaction strengths. The exclusion of simultaneous collisions allows us to analyse individual root-wall collisions locally and compare the corresponding squared projections with time-changed squared Bessel processes. This connects the geometry of the collision-time set to the balance between repulsion and diffusion in the direction normal to each wall. In particular constant-coefficient cases, the upper and lower bounds coincide and yield an exact dimension, as illustrated by the examples in Section~\ref{sec:examples}.

A related direction concerns collisions of eigenvalues and singular values of Gaussian matrix processes and fields. Yuan \cite{Yuan24} establishes collision criteria and Hausdorff dimension results in this setting, including models with fractional Brownian entries. These results highlight the role of the temporal regularity of the driving noise. In the present paper the driving noise remains Brownian, and we study how collision behaviour changes with the coefficients of the singular stochastic differential equation. Fractional Brownian models with Hurst parameter different from $1/2$ are outside the It\^o framework used here.

The statistical-mechanical interpretation is particularly transparent in type $A_{N-1}$. With constant diffusion, constant coupling $k$, and $b=-V'$, the drift is the negative gradient of
    \begin{equation*}
        U(x)=\sum_{i=1}^N V(x_i) -k\sum_{1\leq i<j\leq N}\log(x_j-x_i).
    \end{equation*}
This is the energy of a one-dimensional logarithmically repelling particle system in an external potential. The Hausdorff dimension of collision times describes a feature of its dynamics that complements information about spatial distributions. Our bounds identify conditions under which this temporal behaviour persists when the diffusion intensity and interaction strengths are no longer constant.

The paper is structured as follows.
We begin by recalling several notions and setting notations necessary for our results in Section~\ref{sec:prelim}.
The proof of Theorem~\ref{th:multiplecollisions} is given in a series of propositions detailed in Section~\ref{sec:collisions}.
The proofs of Theorems~\ref{th:Hausdorffupper} and \ref{th:Hausdorfflower}, and Lemma~\ref{th:Hausdorfflowerweak} are given in Section~\ref{sec:dimensionbounds}.
We end the paper with several examples of particle systems for which our results are applicable in Section~\ref{sec:examples}.

\section{Setting, definitions, and properties}\label{sec:prelim}

\subsection{Reflections, root systems, and Weyl chambers}
\label{sec:roots}
To understand the general particle system, we first introduce the key concepts that occur in \eqref{eq:x:SDE} and \eqref{eq:x:SDE:mult}. For any nonzero vectors $y$ and $z\in\R^N$, we define the \emph{reflection operator} $\reflect{y}$ acting on $z$ by the formula
\[
\reflect{y} z := z - 2\frac{\yz}{\yy}y,
\]
which yields the reflection of the vector $z$ across the hyperplane through the origin orthogonal to $y$. 
We will use the notation $|y|=\sqrt{\inner{y}{y}}$ for the Euclidean vector norm of $y\in\R^N$, which is naturally reduced to the absolute value function if $y$ is a scalar. A \emph{root system} $R$ is then defined as a finite set of nonzero vectors, called \emph{roots}, satisfying the condition that for any $\alpha, \beta \in R$, the reflection $\reflect\alpha \beta$ also belongs to $R$; in other words, the reflection of any root along another root is again a root. Furthermore, we require that the root system be \emph{reduced}, that is, for every $\alpha\in R$ we have $\R \alpha\cap R =\{\pm\alpha\}$. The reflections generated by the roots in \( R \) form a \textit{reflection group}, denoted by \( G \).

Some structural properties of reduced root systems will play an important role in the upcoming sections. Since $\alpha\in R$ implies $-\alpha \in R$ a \emph{positive subsystem} \( R_+ \subset R \) can be selected by choosing an arbitrary vector \( u \in \R^N \) such that \( \inner{\alpha}{u} \neq 0 \) for all \( \alpha \in R \). This choice induces a decomposition of \( R \) into positive and negative roots, where
\[
R_+ := \{ \alpha \in R \mid \inner{\alpha}{u} > 0 \}.
\]
The set of negative roots is given by the \emph{negative subsystem} \( -R_+ \). In the following, we omit the dependence on \( R \). For a fixed positive subsystem \( R_+ \), the associated \emph{Weyl chamber} \( \C\) is
\[
\C := \left\{ y \in \R^N \,\middle|\, \inner{\alpha}{y} > 0 \ \text{for all } \alpha \in R_+ \right\}.
\]

We denote by $\SimplePlus$ the \emph{simple system} corresponding to \( R_+ \), consisting of a set of \emph{simple roots}. A positive root is called simple if it cannot be written as the sum of two positive roots. The corresponding simple system $\Delta_+$ forms a basis of $\operatorname{span}R$; see \cite[Theorem, p. 8]{Humphreys}. 
We state a few properties for the simple roots: 
\begin{enumerate}[label=(S\arabic*), leftmargin=2cm]
    \item \label{prop:simpleroots:linearcombination}
    For every root $\alpha\in R_+$ there exists unique $c_\gamma\geq 0$ such that 
    \begin{align*}
    \alpha = \sum_{\gamma \in \SimplePlus} c_\gamma \gamma.
    \end{align*}
    Every root in \( -R_+ \) is a linear combination of simple roots with non-positive coefficients.
    \item \label{prop:simpleroots:scalarproduct}For every $\alpha,\beta\in\SimplePlus$ with $\alpha\not= \beta$ we have $\ab\leq 0$.
\end{enumerate}
The Properties \ref{prop:simpleroots:linearcombination} and \ref{prop:simpleroots:scalarproduct} can be found in \cite[Theorem, p.~8; Corollary, p.~9]{Humphreys}.

For the systems considered in this paper, we choose in the $A_{N-1}$ case
    \begin{align*}
        R_+ &=\{e_j-e_i:1\leq i<j\leq N\},\\
        W&=\{x:x_1<\cdots<x_N\},\\
        \Delta_+ &=\{e_{i+1}-e_i:1\leq i<N\},
    \end{align*}
    where for $B_N$ we have
    \begin{align*}
        R_+&=\{e_i:1\leq i\leq N\} \cup\{e_j-e_i,e_j+e_i:1\leq i<j\leq N\},\\
        W&=\{x:0<x_1<\cdots<x_N\},\\
        \Delta_+ &=\{e_1\}\cup\{e_{i+1}-e_i:1\leq i<N\}.
    \end{align*}
    Finally, in $D_N$ case we consider
    \begin{align*}
        R_+&=\{e_j-e_i,e_j+e_i:1\leq i<j\leq N\},\\
        W&=\{x:|x_1|<x_2<\cdots<x_N\},\\
        \Delta_+ &=\{e_1+e_2\}\cup\{e_{i+1}-e_i:1\leq i<N\}.
    \end{align*}
For more details on the theory of root systems and their properties, see \cite{Humphreys}.

%%%%%%%%%%%%%%%%%%%%%%%%%%%%%%%%%%%%%%%%%%%%%%%%%%%%%%%%%%%%%%%%%%%%%%%%%%%%%%

\subsection{Symmetric polynomials}

For a given vector of variables $A=(a_1,\ldots,a_N)$, we denote by $e_n(A)$ the basic symmetric polynomial in $A$, where $n=1,\ldots,N$. More precisely, we have

$$
	e_n(A) := \sum_{i_1<\ldots<i_n} a_{i_1}\cdot\ldots\cdot a_{i_n}\/,\quad n=1,\ldots,N\/.
$$
In particular, $e_1(A) = a_1+\ldots+a_N$ and $e_{N}(A) = a_1\cdot\ldots\cdot a_N$. For completeness, we set $e_0(A)\equiv 1$ and $e_{n}(A)\equiv 0$ for every $n<0$ or $n>N$. Moreover, for any fixed collection $a_{j_1}, \ldots, a_{j_k}$ of entries of $A$, we denote by
$$
	e^{\overline{a_{j_1}},\ldots,\overline{a_{j_k}}}_n(A) :=  \sum_{\stackrel{i_1<\ldots<i_n}{i_l\neq j_s}} a_{i_1}\cdot\ldots\cdot a_{i_n}\/, \quad n=1,\dots, N-k
$$
which is a sum of all products of length $n$ in which there are no elements of $\{a_{j_1}, \ldots, a_{j_k}\}$. 
For every collection of $k$ distinct deleted indices, we use the conventions
    \begin{equation*}
        e_0^{\overline{a_{j_1}},\ldots,\overline{a_{j_k}}}(A)=1, \qquad
        e_n^{\overline{a_{j_1}},\ldots,\overline{a_{j_k}}}(A)=0 \quad\text{if }n<0\text{ or }n>N-k.
    \end{equation*}

\medskip

\noindent The primary tools for studying the collisions of the process $x$ with the boundary of the Weyl chamber are the symmetric polynomials in $\xa$ for $\aR$. Therefore, we introduce the following notation.
$$
	\Aaclear := \{\xa^2:\aR\}\/.
$$
Moreover, for a given set of positive numbers, which we will call weights
$$
	w := \{w_\alpha>0: \aR\}
$$
we write $\alpha^*$ for a root normalized by $w_\alpha$, that is, $\alpha^* = \alpha/w_\alpha$ and consider the corresponding set 
$$
	\Aa := \{\xas^2:\aR\}\/.
$$
These collections are understood as families indexed by $R_+$: equal numerical values corresponding to different roots are retained as separate entries. We write $M=|R_+|$ for the number of entries, independently of $x$ and of the positive weights. For example, for the $\AN$ root system we have $M=N(N-1)/2$.

We take into consideration the corresponding symmetric polynomials in $\Aa$. We also simplify the notation and write $\eaA{n}$ for the symmetric polynomial of degree $n$ which does not include $\xas^2$. We write similarly $\eabA{n}$ whenever we want to exclude $\xas^2$ and $\xbs^2$. 

\bigskip

We end this subsection with two technical lemmas that provide algebraic formulas used in what follows. 

\begin{lemma}
For every $n = 1,\ldots, M$ and every $\alpha,\beta\in R_+$ such that $\alpha\neq \beta$ we have
\begin{equation}
	\label{eq:e:form2}
	\eabA{n-2}\eA{n} = \eaA{n-1}\ebA{n-1}+\eabA{n}\eabA{n-2}-\big(\eabA{n-1}\big)^2.
\end{equation}
\end{lemma}
This identity holds for elementary symmetric polynomials in any finite indexed family of variables; its proof does not use the root-system structure.

\begin{proof}
    We have for every $\alpha\neq \beta$ that
    \begin{equation*}
        \eA{n} = \xas^2 \eaA{n-1}+\eaA{n}\/,\quad \ebA{n-1} = \xas^2\eabA{n-2}+\eabA{n-1}\/,
    \end{equation*}
    which gives
    \begin{align*}
        \eabA{n-2}\eA{n} &= \eabA{n-2}\left( \xas^2 \eaA{n-1}+\eaA{n}\right)\\
        &=  \xas^2\eabA{n-2}\eaA{n-1}+\eabA{n-2}\eaA{n}\\
        &= \left(\ebA{n-1}-\eabA{n-1}\right)\eaA{n-1}+\eabA{n-2}\eaA{n}\\
        &= \eaA{n-1}\ebA{n-1}+\eabA{n-2}\eaA{n} - \eabA{n-1}\eaA{n-1}.
    \end{align*}
    Since $\eaA{n} = \xbs^2\eabA{n-1}+\eabA{n}$ and $\eaA{n-1} = \xbs^2\eabA{n-2}+\eabA{n-1}$ we finally get
    \begin{align*}
        \eabA{n-2}\eaA{n} - \eabA{n-1}\eaA{n-1} &= \eabA{n-2}\left(\xbs^2\eabA{n-1}+\eabA{n}\right)- \eabA{n-1}\eaA{n-1}\\
       &= \eabA{n-2}\eabA{n} - \eabA{n-1}\left(\eaA{n-1}-\xbs^2\eabA{n-2}\right)\\
       &= \eabA{n-2}\eabA{n}- \big( \eabA{n-1}\big)^2.
    \end{align*}

\end{proof}
\begin{lemma}
\label{prop:AP:alfagamma}
   If $\ab\neq 0$ and $\alpha\neq\beta$, denote by $\gamma$ one of $\pm\reflect{\beta}\alpha=\pm(\alpha-2\beta\ab/|\beta|^2)$, which is in $R_+$. Then, we have
	\begin{equation}
        \label{eq:ac:form1}
            \ab\xa+\bc\xc = \frac{2\xb}{|\beta|^2}
	\end{equation}
    and
    	\begin{equation}
        \label{eq:ac:form2}
            \ab\xc+\bc\xa = \frac{2\ab\bc\xb}{|\beta|^2}
	\end{equation}
\end{lemma}
\begin{proof}
   Since the expression $\bc\xc$ is the same for $\gamma$ and $-\gamma$ and the equality \ref{eq:ac:form2} is invariant under the sign change of $\gamma$, we can assume that 
	\begin{equation*}
	  \gamma = \reflect{\beta}\alpha= \alpha-\frac{2\beta\ab}{|\beta|^2}\/,
	\end{equation*}
	which leads to $\bc=\ab-2\ab=-\ab$ and hence
    \begin{align*}
        \bc\xc &= -\ab\left(\xa-\frac{2\xb\ab}{|\beta|^2}\right) = -\ab\xa + \frac{2\ab^2\xb}{|\beta|^2}\/,\\  
        \ab\xc &= \ab \left(\xa-\frac{2\xb\ab}{|\beta|^2}\right) = -\bc\xa + \frac{2\ab\bc\xb}{|\beta|^2}
    \end{align*}
     and the results follow from $\ab^2=1$, see \ref{ass:R:root system}.
\end{proof}

\subsection{Hausdorff dimension}
The Hausdorff dimension generalizes the familiar notion of dimension. This means that well-known geometric objects like straight lines, hyperplanes and others with intuitive dimensionality keep the same dimension. The Hausdorff dimension offers a finer distinction, since it admits positive real numbers. 
	
	We denote by $B(y,r):=\{z\in \R^N:|y-z|\leq r\}$ the closed $N$-dimensional ball centered at $y$ with radius $r$. For the monotonically increasing monomial (on the positive halfline) of power $\varkappa\geq 0 $ the Hausdorff measure is specified by 
	\begin{align*}
		m_\varkappa(E):=\lim\limits_{\varepsilon\to 0} \inf\bigg\{\sum_{i=1}^\infty (2r_i)^\varkappa\cond{\big}\exists\ 0\leq r_i<\varepsilon,\, y_i\in \R^n: E\subset \bigcup\limits_{i=1}^\infty B(y_i,r_i)\bigg\}.
	\end{align*}
	The Hausdorff dimension is defined by the following lemma, see \cite[8.1 Hausdorff dimension]{Adler1981}. 
	\begin{lemma}
		\label{lm:Hausdorff_dimension}
		For any set $E\subset \mathbb{R}^n$ there exists a unique number $d\in[0,n]$, called the Hausdorff dimension of $E$, for which 
		\begin{align*}
			\varkappa <d\Rightarrow m_\varkappa(E)=\infty, \quad \varkappa >d\Rightarrow m_\varkappa(E)=0. 
		\end{align*}
		This number is denoted by $\dim (E)$ and satisfies
		\begin{align*}
			d=\dim(E)=\sup\{\varkappa>0:m_\varkappa(E)=\infty \}=\inf \{\varkappa>0:m_\varkappa(E)=0\}.
		\end{align*}
	\end{lemma}
The Hausdorff dimension satisfies the following properties.
    \begin{enumerate}[label=(H\arabic*), leftmargin=2cm]
		\item \textit{Countable stability:} \label{prop:Hausdorff:countablestability}If $F_1,F_2,\dots $ is a (countable) sequence of sets then $$\dim \Big(\bigcup\limits_{i=1}^\infty F_i\Big)=\sup\limits_{i\in\mathbb{N}}\dim (F_i).$$
		\item\textit{Monotonicity:} \label{prop:Hausdorff:monotonicity} If $E\subset F$ then $\dim (E)\leq \dim (F)$.
        \item\textit{Invariance under bi-Lipschitz mapping:}\label{prop:Hausdorff:invariance:bilipschitz} If $f:E\to \R^N$ is a bi-Lipschitz mapping, that is, if there exist constants $c_2\geq c_1>0$ such that for every $y,z\in E$ 
        $$c_1|y-z|\leq |f(y)-f(z)|\leq c_2|y-z|,$$
        then $\dim(E)=\dim(f(E))$.
	\end{enumerate}
Properties \ref{prop:Hausdorff:countablestability} and \ref{prop:Hausdorff:monotonicity} can be found in \cite[Section 2.2, Hausdorff Dimension]{Falconer2013fractal}, while \ref{prop:Hausdorff:invariance:bilipschitz} is stated in \cite[Corollary 2.4]{Falconer2013fractal}.
\subsection{Additional notation}
We will write  
    \begin{equation*}
        \SanbR\qquad \textrm{and} \qquad \SabR
    \end{equation*}
for the off-diagonal sum and the double sum, respectively. For given $\beta\in R_+$ we introduce the following notation 
\begin{align}
    \label{eq:Rplus}
    R_+(\beta) = \big\{(\alpha,\gamma)\cond{} \alpha,\gamma\in R_+, \alpha\neq \beta, \ab\neq 0,\gamma = \pm\reflect{\beta}\alpha\big\}.
\end{align}
For example, for $i<j<k$ and $\beta=e_j-e_i$,
the pair $(e_k-e_i,e_k-e_j)$ belongs to $R_+(\beta)$.

\section{Simple and multiple collisions}
\label{sec:collisions}
In this section we study the nature of collisions in more detail presenting a series of propositions leading to the proof of Theorem~\ref{th:multiplecollisions}. 

The general idea including the choice of weights in the proof can be understood through a simple example. Let $U$ and $V$ be nonnegative Bessel-type diffusions satisfying
\begin{equation*}
    dU_t=a\,dB_t^{(1)}+\frac{\nu_1}{U_t}\,dt,\quad 
    dV_t=b\,dB_t^{(2)}+\frac{\nu_2}{V_t}\,dt,
\end{equation*}
where $a,b,\nu_1,\nu_2>0$ and the Brownian motions are independent. Each process may hit zero: this happens when $0<\nu_1<a^2/2$ and $0<\nu_2<b^2/2$, respectively. Nevertheless, they do not hit zero simultaneously at a positive time.
Indeed, consider the weighted sum
    \begin{equation*}
        Z_t=\frac{U_t^2}{a^2}+\frac{V_t^2}{b^2}.
    \end{equation*}
It\^o's formula gives
    \begin{equation*}
        dZ_t = 2\frac{U_t}{a}\,dB_t^{(1)} +2\frac{V_t}{b}\,dB_t^{(2)} +\delta\,dt, \qquad \delta=2+\frac{2\nu_1}{a^2}+\frac{2\nu_2}{b^2}>2.
    \end{equation*}
The quadratic variation of its martingale part is $4Z_t\,dt$. Consequently, there exists a standard Brownian motion $\widetilde B$ with
    \begin{equation*}
        dZ_t=2\sqrt{Z_t}\,d\widetilde B_t+\delta\,dt.
    \end{equation*}
Thus, $Z$ is a squared Bessel process of dimension greater than two and is strictly positive at every positive time. Since $Z_t=0$ is equivalent to $U_t=V_t=0$, simultaneous zeros are excluded. When $a=b$, the same argument uses $U_t^2+V_t^2$ up to a constant factor; for unequal coefficients, the inverse-variance weights give the appropriate normalization.

This also explains the logarithmic functional used below.
Before a possible zero of $Z$,
\begin{equation*}
    d(-\log Z_t)
    =-\frac{2}{\sqrt{Z_t}}\,d\widetilde B_t
     -\frac{\delta-2}{Z_t}\,dt.
\end{equation*}
The singular drift has the sign that prevents explosion
of $-\log Z$ to $+\infty$.

For the particle system, we apply the same normalization
locally near a fixed boundary point $z$, choosing
\begin{equation*}
    w_\alpha^2
       =\sum_{i=1}^N\alpha_i^2\sigma^2(z_i).
\end{equation*}
The weighted symmetric polynomials detect simultaneous vanishing of root projections, and their zero sets do not depend on the positive weights. Unlike the two processes in the example, the root projections generally have correlated martingale parts. The reflection identities control these correlations, while continuity allows us to freeze the variances at $z$. The resulting errors are absorbed by the strictly positive repulsive terms.

For fixed $y\in \Cc$, we introduce two sets indicating the roots that make $\inner{y}{\alpha}$ equal to zero and those for which this expression is positive, respectively,
  \begin{equation*}
		\Pzero = \{\aR: \inner{\alpha}{y}=0\}\/,\quad           \Pplus = \{\aR: \inner{\alpha}{y}>0\}\/.
  \end{equation*}
Note that $\Pzero \cup \Pplus = R_+$ and $\Pzero \cap \Pplus = \emptyset$ for every $y\in\Cc$. Obviously $\Pzero = \emptyset$ and $\Pplus=R_+$ for every $y\in \C$ and $\Pzero$ becomes nonempty on the boundary $\Cb$. We split the boundary $\Cb$ of the Weyl chamber $\C$ into parts depending on the number of elements in $\Pzero$. More precisely, we introduce the following.
\begin{definition}
	We say that $y\in \Cb $ is \textit{a collision point of order $m$}, where $m\in\{1,\ldots,M\}$, if $|\Pzero|=m$. We denote the set of all collision points of order $m$ as $\Cbm$.
\end{definition}

\begin{definition}
	We say that for the general particle system $x = (x_1,\ldots, x_N)$ a collision of order $m$ occurs at time $t$, if $x_t\in\Cbm$. If $m=1$ we call it \textit{a single collision} and \textit{a multiple collision} is a collision of order $m\geq 2$.
\end{definition}
Note that the boundary of the Weyl chamber $\Cb$ decomposes into a disjoint union of $\Cbm$ over $m\in\{1,\ldots,M\}$. Moreover, collisions of $x=(x_1,\ldots,x_N)$ and their orders are controlled by the symmetric polynomials $\eA{n}$. In fact, there is a collision of order $m$ at time $t$ if and only if $\eA{n}_t = 0$ and $\eA{n-1}_t > 0$, where $n= M-m+1$. 
Let us denote the corresponding first hitting times of zero
for $\eA{n}$ by
    \begin{equation*}
        \tau_n :=\inf\bigl\{t\in(0,T_\infty):\eA{n}_t=0\bigr\} \wedge T_\infty, \qquad n=0,\ldots,M,
    \end{equation*}
where $\inf\varnothing=\infty$.

Note that the hitting times introduced above are the same for every choice of the weights $w$ due to their positivity. Consequently, we omit $w$ in the notation here. Obviously, the hitting times are ordered as follows
$$
	\tau_{M}\leq \tau_{M-1}\leq \ldots\leq \tau_1\leq \tau_0=\lifetime\/.
$$
Here, we took advantage of the convention that $\eA{0}\equiv 1$, which immediately gives $\tau_0=\lifetime$. The order counts vanishing positive-root projections. In type $A_{N-1}$, if the clusters of equal coordinates have sizes $r_1,\ldots,r_\ell$, then
    \begin{equation*}
        m=\sum_{j=1}^{\ell}\binom{r_j}{2}.
    \end{equation*}
Thus, a triple particle collision has order three, whereas two simultaneous disjoint pair collisions have order two. When $M=1$, the assertion excluding simultaneous root-wall collisions is vacuous.

Our starting point is the stochastic description of polynomials $\eA{n}$, which is provided in the following proposition.

%%%%%%%%%%%%%%%%%%%%%%%%%%%%%%%%%%%%%%%% PROPOSITION 1 %%%%%%%%%%%%%%%%%%%%%%%%%%%%%%%%%%%%%%%%%%%%%%%%%%%%%%%%%%%%%%%%%%%%%%%%%%%%%%%%%%

\begin{proposition}
	\label{prop:eA:SDE}
	For every fixed set of weights $w = \{w_\alpha>0: \alpha\in R_+\}$ and $n=1,\ldots,M$ we have
	\begin{eqnarray*}
	d\eA{n} &=& 2\sum_{k=1}^N\bigg(\sum_{\aR} \alpha_k^* \xas \eaA{n-1}\bigg)\sigma(x_k)dB_k+2\sum_{k=1}^N b(x_k)\sum_{\aR}\alpha_k^* \xas \eaA{n-1}dt\\
	&&+ 2\SabR \frac{\abs\xas}{\xbs} \eaA{n-1}k_\beta(x)dt+\sum_{k=1}^N\sigma^2(x_k)\sum_{\aR}(\alpha_k^*)^2\eaA{n-1}\,dt\\
	&&+ 2\sum_{k=1}^N \sigma^2(x_k)\SanbR \alpha_k^*\beta_k^* \xas \xbs \eabA{n-2} dt
\end{eqnarray*}
up to the lifetime $\lifetime$.
\end{proposition}
\begin{proof}
Using the fact that $dx_kdx_j = 0$, whenever $k\neq j$, the It\^o formula leads to 
\begin{equation*}
     d\eA{n} =   \sum_{k=1}^N \dfrac{\partial}{\partial x_k}\, \eA{n}dx_k + \dfrac12 \sum_{k=1}^N  \dfrac{\partial^2}{\partial x_k^2}\, \eA{n} dx_kdx_k\/.
\end{equation*}
Since we simply have the following representations of the derivatives
\begin{eqnarray*}
    \dfrac{\partial}{\partial x_k}\, \eA{n} & = & 2 \sum_{\aR} \alpha_k^* \xas\eaA{n-1}\/,\\
    \dfrac{\partial^2}{\partial x_k^2}\, \eA{n} & = & 2 \sum_{\aR} (\alpha_k^*)^2 \eaA{n-1} + 4 \SanbR \alpha_k^* \beta_k^* \xas \xbs \eabA{n-2} \/,
\end{eqnarray*}
the result follows now directly from \eqref{eq:x:SDE}.
\end{proof}

\bigskip

Let us assume that $\eA{M}_0>0$, that is, we start from the interior of the Weyl chamber. We consider a collection of semi-martingales
$$
	S_n^w(t) := -\ln \eA{n}_t\/,\quad n=1, 2,\ldots,M
$$
for $t<\tau_{n}$. These processes control the first hitting times $\tau_n$ in the way that $S_n^w$ explodes to $+\infty$ before the life-time $\lifetime$ if and only if $\eA{n}$ hits $0$. To explore these relations, we have to start with the following stochastic description of $S_n^w$.

%%%%%%%%%%%%%%%%%%%%%%%%%%%%%%%%%%%%%%%% PROPOSITION 2 %%%%%%%%%%%%%%%%%%%%%%%%%%%%%%%%%%%%%%%%%%%%%%%%%%%%%%%%%%%%%%%%%%%%%%%%%%%%%%%%%%
\begin{proposition}
    \label{prop:Sn:SDE}
    Assume that $\eA{M}_0>0$. For every fixed set of weights $w = \{w_\alpha>0: \alpha\in R_+\}$ we have the semi-martingale decomposition $dS_n^w = dM_n^w+dA_n^w$, where the local martingale part is
    \begin{equation*}
	dM_n^w = -\dfrac{2}{\eA{n}}\sum_{k=1}^N\left(\sum_{\aR} \alpha_k^* \xas \eaA{n-1}\right)\sigma(x_k)dB_k
    \end{equation*}
 and the drift part $dA_n^w$ is given as the sum $(A_1+A_2+A_3+A_4+A_5+A_6)\,dt$ of the following six components 
    \begin{eqnarray}
        \label{S:A1:form}
	A_1 &=& \frac{1}{(\eA{n})^2}\sum_{k=1}^N \sigma^2(x_k)\SaR (\alpha_k^*)^2  \left(\xas^2\eaA{n-1}-\eaA{n}\right)\eaA{n-1}\/,\\
        \label{S:A2:form}
        A_2 &=& \frac{2}{(\eA{n})^2}\sum_{k=1}^N \sigma^2(x_k)\SanbR \alpha_k^* \beta_k^* \xas\xbs\big(\eabA{n-1}\big)^2\/,\\
        \label{S:A3:form}
        A_3 &=& -\frac{2}{(\eA{n})^2}\sum_{k=1}^N \sigma^2(x_k)\SanbR \alpha_k^* \beta_k^* \xas\xbs\eabA{n}\eabA{n-2}\/,
     \end{eqnarray}  
     and
     \begin{eqnarray}
	\label{S:A4:form}
        A_4 &=& - \frac{2}{\eA{n}}\sum_{k=1}^N b(x_k)\sum_{\aR}\alpha_k^* \xas \eaA{n-1}\/,\\
        \label{S:A5:form}
	A_5 &=&  - \frac{2}{\eA{n}}\sum_{\aR} |\alpha^*|^2\eaA{n-1}k_\alpha(x)\/,\\
        \label{S:A6:form}
	A_6 &=&  	- \frac{2}{\eA{n}}	\SanbR \frac{\abs\xas}{\xbs} \eaA{n-1}k_\beta(x)\/,	
    \end{eqnarray}
    for $t<\tau_n \wedge \lifetime$.
\end{proposition}

\begin{proof}
		Applying the It\^o formula for $t<\tau_n\wedge \lifetime
        $ we get
    \begin{equation*}
	dS_n^w = -\frac{d\eA{n}}{\eA{n}}+\dfrac{d\eA{n}d\eA{n}}{2(\eA{n})^2}\/.
    \end{equation*}
		The stochastic description of $\eA{n}$ given in Proposition \ref{prop:eA:SDE} directly gives the form of the martingale part $dM^w_n$, but we can also use it to calculate  
    \begin{eqnarray*}
	\frac12 d\eA{n}d\eA{n} &=& 2\sum_{k=1}^N\left(\sum_{\aR} \alpha_k^* \xas \eaA{n-1}\right)^2\sigma^2(x_k)\,dt\\
	   &=& 2\sum_{k=1}^N \sigma^2(x_k) \SaR (\alpha_k^*)^2 \xas^2 \big(\eaA{n-1}\big)^2\,dt\\
		 && + 2\sum_{k=1}^N \sigma^2(x_k)\SanbR \alpha_k^* \beta_k^* \xas\xbs \eaA{n-1}\ebA{n-1}dt\/.
    \end{eqnarray*} 
    Finally, we can write the drift part of $dS_n^w$ as the sum $(\widetilde{A_1}+\widetilde{A_2}+A_4+A_5+A_6)dt$, where
    \begin{eqnarray*}
        \widetilde{A_1} &=& \frac{1}{(\eA{n})^2}\sum_{k=1}^N \sigma^2(x_k)\SaR (\alpha_k^*)^2  \left(2\xas^2 \eaA{n-1}-\eA{n}\right)\eaA{n-1}\/,\\
	\widetilde{A_2} &=& \frac{2}{(\eA{n})^2}\sum_{k=1}^N \sigma^2(x_k)\SanbR \alpha_k^* \beta_k^* \xas\xbs \left(\eaA{n-1}\ebA{n-1} - \eabA{n-2}\eA{n} \right)
    \end{eqnarray*}
    and $A_4$, $A_5$, $A_6$ are given in \eqref{S:A4:form}, \eqref{S:A5:form}, \eqref{S:A6:form}, respectively. Note that $A_5$ and $A_6$ are obtained by splitting the sum 
		$$
			\sum_{\aR}\sum_{\bR} \frac{\abs\xas}{\xbs} \eaA{n-1}k_\beta(x)
		$$
		into on and off diagonal sums. The formula \eqref{S:A1:form} for $A_1$ is obtained from rewriting the formula $\widetilde{A_1}$ given above by $\eA{n} = \xas^2\eaA{n-1}+\eaA{n}$, which leads to 
    \begin{equation*}
        2\xas^2 \eaA{n-1}-\eA{n} = \xas^2\eaA{n-1}-\eaA{n}\/.
    \end{equation*}
    From the other side, using \eqref{eq:e:form2}, we get
    \begin{equation*}
        \eaA{n-1}\ebA{n-1} - \eabA{n-2}\eA{n} = \big(\eabA{n-1}\big)^2-\eabA{n}\eabA{n-2}
    \end{equation*}
    and we can rewrite $\widetilde{A_2}$ as the sum of $A_2$ and $A_3$ given in \eqref{S:A2:form} and \eqref{S:A3:form}.
\end{proof}

In the next proposition, we show that double collisions do not occur up to the life time $\lifetime $, whenever the system starts from the interior of the Weyl chamber. This is the most crucial and difficult part of the proof of Theorem \ref{th:multiplecollisions}. In the base case, where $\sigma$ is constant and equal to $1$, the proof is reduced to showing that all components of the drift of process $S$ are bounded or finite for every finite $t$, which prevents the process from exploding to infinity in finite time. When the drift parts are modified by a positive $\sigma$, the matter becomes more subtle and requires, on one hand, localization — an analysis of the drift of $S$ in the vicinity of a fixed boundary point — and, on the other hand, the proper selection of weights to achieve a similar effect as in the $\sigma = 1$ case. The key here is that the explosion of $S^w$ for fixed weight $w$ is equivalent to an explosion for any other positive weight, and thus also for the process without weights.

%%%%%%%%%%%%%%%%%%%%%%%%%%%%%%%%%%%%%%%% PROPOSITION 5 %%%%%%%%%%%%%%%%%%%%%%%%%%%%%%%%%%%%%%%%%%%%%%%%%%%%%%%%%%%%%%%%%%%%%%%%%%%%%%%%%%

\begin{proposition}
    Assume that $\eA{M}_0>0$, that is, the particle system $x$ starts from the interior of the Weyl chamber. Then, under the assumptions of Theorem \ref{th:multiplecollisions}, we have
    \begin{equation*}
        \tau_{M-1} = \tau_{M-2} = \ldots = \tau_1 = \tau_0= \lifetime
    \end{equation*}
    almost surely.
\end{proposition}

\begin{proof}
        We fix $n\in\{1,\ldots,M-1\}$ and work on the set $\{\tau_{n}<\tau_{n-1}\leq\lifetime\}$ with the intention of showing that its probability is zero.
		Note that on $\{\tau_{n}<\tau_{n-1}\}$ we have $\eA{n}(\tau_{n})=0$ and $\eA{n-1}(\tau_{n})>0$, which is equivalent to saying $x_{\tau_n} \in \Cbm$, where $m=M-n+1$.

        \medskip

		We begin by constructing the following open cover of $\Cbm$. Let us fix $y\in \Cbm$. The continuity argument enables us to find a bounded neighbourhood, open relative to $\overline W$, $E=E(y):=\{x\in\overline W:|x-y|<r\}$, such that the following statements hold.
        
		\begin{enumerate}[label=(\roman*)]
			\item \label{eq:Pzero:inclusion}For every $x\in E$ we have 
				\begin{equation*}
				%\label{eq:Pzero:inclusion}
					\Pzerox \subseteq \Pzero\/.
				\end{equation*}
			\item \label{eq:Hplus:condition}
                For every $x\in E$ we have
			  \begin{equation*}
					%\label{eq:Hplus:condition}
					\prod_{\alpha\in \Pplus} \xa^2>0.
				\end{equation*}   
			\item \label{eq:k:lowerbound} For every $x\in E$ and every $\alpha\in R_+$ we have
				\begin{equation*}
                    %\label{eq:k:lowerbound}
					k_\alpha(x)\geq \varepsilon
				\end{equation*}
				where $\varepsilon = \varepsilon(y) = \frac{1}{2}\inf_{\aR} k_\alpha(y)>0$.
            \item \label{eq:sigma:uniform} For all $x,z\in E$ and $k=1,\ldots,N$ we have
                    \begin{equation*}
                        %\label{eq:sigma:uniform}
                        |\sigma^2(x_k)-\sigma^2(z_k)|\leq \frac{\varepsilon}{8M}\/.
                    \end{equation*}

			 \item 
              \label{eq:sigma2:ratio} For all $x\in E$ and $\alpha\in \Pzero$ we have
				  \begin{equation*}
						%\label{eq:sigma2:ratio}
					    1-\frac{\varepsilon_0}{8M} \leq \frac{\sum_{k=1}^N \sigma^2(x_k)\alpha_k^2}{\sum_{k=1}^N \sigma^2(y_k)\alpha_k^2} \leq 1+\frac{\varepsilon_0}{8M}\/,
				  \end{equation*}		
                where $\varepsilon_0 = \varepsilon \min_{\alpha\in \Pzero} |\alpha^*|^2$ 
		\end{enumerate}
    Let us explain why such a neighbourhood can be chosen. For fixed $y$, the weights appearing in \ref{eq:sigma2:ratio} are those defined in \eqref{eq:weights}, with the variances evaluated at $y$. Thus, these weights, and consequently $\varepsilon_0$, depend only on $y$ and are fixed independently of the radius of the neighbourhood. Their positivity, as well as positivity of $\varepsilon$, follows from \ref{ass:sigma:k:positive}. Continuity now gives a neighbourhood on which \ref{eq:Hplus:condition}--\ref{eq:sigma2:ratio} hold. By decreasing its radius if necessary, we may choose $E$ so that these conditions hold also on $\overline E\cap\Cc$. In particular,
        \begin{equation*}
            \inf_{x\in\overline E\cap\Cc} \prod_{\alpha\in\Pplus}\langle x,\alpha\rangle^2 >0.
        \end{equation*}
    Throughout this localization, conditions involving points of $E$ are understood on $E\cap\Cc$, where the coefficients are defined. Although \ref{eq:Pzero:inclusion} follows from \ref{eq:Hplus:condition}, we state it separately to make the geometric property explicit and facilitate references to it in the proof.
    
    From the above given cover $\big\{E(y): y\in \Cbm\big\}$ we can select a countable sub-cover since we are working on a separable metric space. Let us denote this countable sub-cover of balls by $\{E_i: i\in \N\}$, where $E_i=E(y^i)$ for some $y^i\in \Cbm$. Consequently, we get
		$$
			\{\tau_{n}<\tau_{n-1}\} = \bigcup_{i\in\N}\{\tau_{n}<\tau_{n-1},x_{\tau_n} \in E_i\}
		$$
		and we can restrict our consideration to $\{\tau_{n}<\tau_{n-1},x_{\tau_n} \in E_i\}$ for fixed $i$ and show that this set has probability zero, as then we deal with a countable collection of sets of measure zero. 
		
		Our main tool here is the process $S_n^w$ with a suitable chosen set of weights and the fact that it is enough to find one $S_n^w$ which cannot explode on $E_i$ under the additional condition $\tau_{n}<\tau_{n-1}$ to conclude that $\pr(\tau_{n}<\tau_{n-1},x_{\tau_n} \in E_i)=0$. Consequently, the crucial issue is the choice of a suitable set of weights. It turns out that the desired weights are determined by the martingale coefficient $\sigma^2$ in the following way
        \begin{align}
            \label{eq:weights}
            w_\alpha = \left(\sum_{k=1}^N \alpha_k^2\sigma^2(y^i_k)\right)^{1/2}\/,\quad \alpha\in R_+\/.
        \end{align}
		  Positivity of the weights follows from \ref{ass:sigma:k:positive}. From now on, we consider $w=\{w_\alpha: \alpha\in R_+\}$, with $w_\alpha$ defined above. Conditions \ref{eq:Pzero:inclusion} and \ref{eq:Hplus:condition} ensure that for every collision point $y$ of order $m=M-n+1$ from $E_i$ we have $\Pzero = P^0(y^i)$. 
        Thus, for every $y\in \Cbm\cap E_i$ we define the function
		\begin{equation*}
			H_+^w(x) = \prod_{\alpha\in \Pplus} \xas^2\/,\quad x\in \C\/,
		\end{equation*}
		which does not depend on $y$ as long as $y\in \Cbm\cap E_i$ and is strictly positive on $E_i$ since 
		\begin{equation*}
			H_+^w(x) = \prod_{\alpha\in \Pplus}\frac{\xa^2}{w_\alpha^2} = \left(\prod_{\alpha\in \Pplus}\frac{1}{w_\alpha^2}\right) \prod_{\alpha\in \Pplus}\xa^2\/\stackrel{\text{\ref{eq:Hplus:condition}}}{>}0.
		\end{equation*}
		We will often use the positivity of $H^w_+(x)$ together with the following simple bound
    \begin{equation}
    \label{eq:ineq_en_K}
        \eA{n} \geq  H^w_+(x) \sum_{\alpha\in \Pzero} \xas^2 \/. 
    \end{equation}
	For the purposes of this proof, as we are now restricted to the set $E_i$, we introduce the following auxiliary notation
    \begin{equation*}
	c_1 = \sup_{x\in E_i\cap\Cc, \aR}\sum_{k=1}^N |\alpha_k^*\, b(x_k)|\/,\quad c_2 = \sup_{x\in E_i\cap\Cc,\alpha\in \Pplus} \frac{1}{\xas}\/,\qquad 
    c_3 = \sup_{x\in E_i\cap\Cc} \frac{1}{H_+^w(x)}\/.
    \end{equation*}

    If $\Pplus=\varnothing$, we use the conventions $H_+^w(x)=1$ and $c_2=0$. These constants are finite by continuity on the compact set $\overline{E_i}\cap\Cc$ and the positive lower bounds ensured by our choice of the neighbourhood.

\medskip 

We work on
$\{\tau_n<\tau_{n-1},\,x_{\tau_n}\in E_i\}$. By continuity, there exists $\delta\in(0,\tau_n)$ such that $x_t\in E_i$ for every $t\in[\tau_n-\delta,\tau_n]$. Write
    \begin{equation*}
        a_n^w(t)=A_1(t)+\cdots+A_6(t).
    \end{equation*}
It is enough to prove that
    \begin{equation*}
        \int_{\tau_n-\delta}^{\tau_n} \bigl(a_n^w(s)\bigr)^+\,ds<\infty.
    \end{equation*}
On $[0,\tau_n-\delta]$, the polynomial $\eA{n}$ is bounded away from zero. Continuity of the coefficients and local integrability of the individual singular drifts therefore give
    \begin{equation*}
        \int_0^{\tau_n-\delta}|a_n^w(s)|\,ds<\infty.
    \end{equation*}
Consequently, the positive variation of the drift of $S_n^w$ is finite up to $\tau_n$. If $S_n^w(t)\to+\infty$ as $t\uparrow\tau_n$, then
    \begin{equation*}
        M_n^w(t) =S_n^w(t)-S_n^w(0)-\int_0^t a_n^w(s)\,ds \longrightarrow+\infty.
    \end{equation*}
This is excluded by the McKean argument \cite[Problem 7, p.~47]{McKean}. In the following bounds we consider $t\in[\tau_n-\delta,\tau_n)$ and write $y=x_{\tau_n}$. The sets $P^0(y)$ and $P^0(y^i)$ coincide.

\bigskip

%%%%%%%%%%%%%%%%%%%%%%%%%%%%%%%%%%%%% STEP 1

{\it Step 1.} We begin with the first two parts $A_1$ and $A_2$ given in \eqref{S:A1:form} and \eqref{S:A2:form}, respectively. Note that if in any component of the sum in \eqref{S:A1:form} we find two factors that go to zero, that is, the product $\xas^2\xbs^2$ with $\alpha,\beta\in \Pzero$ ($\alpha$ could be equal to $\beta$),
then this component divided by $(\eA{n})^2$ is generally bounded by $[H^w_+(x)]^{-2}$ and then by $c_3^2$ using \eqref{eq:ineq_en_K}.
This allows us to focus only on the components where elements $\xas^2$ with $\alpha\in\Pzero$ occur individually. Indeed, for every fixed $\alpha\in \Pplus$, in every component of $\eaA{n-1}$ and $\eaA{n}$, we can always find $\xbs^2$ such that $\beta\in \Pzero$. Therefore, in products of the form $\eaA{n-1}\eaA{n-1}$ and $\eaA{n-1}\eaA{n}$, we will always find a product of two such vanishing expressions, which, in light of the previous observation, makes the sum over $\Pplus$ in \eqref{S:A1:form} bounded from above. Thus, we can just focus on
    \begin{equation*}
        \overline{A_1} = \left(\frac{H^w_+(x)}{\eA{n}}\right)^2 \sum_{k=1}^N \sigma^2(x_k)\SaPz(\alpha_k^*)^2   \bigg(\xas^2-\SbPzna\xbs^2\bigg)\/.
    \end{equation*}
In a similar way, we have to only deal with
\begin{eqnarray*}
	\overline{A_2} &=& 2\left(\frac{H^w_+(x)}{\eA{n}}\right)^2\sum_{k=1}^N \sigma^2(x_k)\mathop{\sum_{\alpha\in\Pzero}\sum_{\beta\in\Pzero}}_{\alpha\neq \beta} \alpha_k^*\beta_k^* \xas\xbs \/.
\end{eqnarray*}

The first thing to do is replace $\sigma^2(x_k)$ with $\sigma^2(y_k)$ in the sums given above using \ref{eq:sigma:uniform}. Since we have exactly $M-n+1$ elements in $\Pzero$, by the Cauchy-Schwarz inequality, we get the bounds
\begin{eqnarray*}
\bigg|\mathop{\sum_{\alpha\in\Pzero}\sum_{\beta\in\Pzero}}_{\alpha\neq \beta} \alpha_k^*\beta_k^* \xas\xbs \bigg|&\leq&
\bigg(\sum_{\alpha\in\Pzero}|\alpha_k^*|\xas\bigg)^2 \leq  (M-n+1) \sum_{\alpha\in\Pzero}|\alpha_k^*|^2\xas^2\/,
\end{eqnarray*}
which together with \ref{eq:sigma:uniform} gives
\begin{eqnarray*}
    \overline{A_2} &\leq&  2\left(\frac{H_+^w(x)}{\eA{n}}\right)^2
\sum_{k=1}^N \sigma^2(y_k)
\mathop{\sum_{\alpha\in\Pzero}\sum_{\beta\in\Pzero}}_{\alpha\neq \beta} \alpha_k^*\beta_k^* \xas\xbs\\
&&+
2\left(\frac{H_+^w(x)}{\eA{n}}\right)^2
\sum_{k=1}^N
|\sigma^2(x_k)-\sigma^2(y_k)|
\left|
\mathop{\sum_{\alpha\in\Pzero}\sum_{\beta\in\Pzero}}_{\alpha\neq \beta} \alpha_k^*\beta_k^* \xas\xbs
\right|\\
    &\leq&\left(\frac{H_+^w(x)}{\eA{n}}\right)^2\Bigg[\sum_{k=1}^N \sigma^2(y_k)\mathop{\sum_{\alpha\in\Pzero}\sum_{\beta\in\Pzero}}_{\alpha\neq \beta} 2\alpha_k^*\beta_k^* \xas\xbs +\frac{\varepsilon}{4}\sum_{\alpha\in\Pzero}|\alpha^*|^2\xas^2\Bigg].
\end{eqnarray*}

Note that for $\alpha\in\Pzero$ we have $\xas^2 H^w_+(x)\leq \eA{n}$ (see \eqref{eq:ineq_en_K}) and $H^w_+(x)\leq \eaA{n-1}$
, which together with the definition of $\varepsilon$ given in \ref{eq:k:lowerbound} gives that 
    \begin{eqnarray*}
        \frac{\varepsilon}{4}\left(\frac{H_+^w(x)}{\eA{n}}\right)^2\sum_{\alpha\in\Pzero}|\alpha^*|^2\xas^2\leq -\frac{1}{8}A_5\/.
    \end{eqnarray*}

For $\overline{A_1}$, put
    \begin{equation*}
        Q=\sum_{\alpha\in\Pzero}\xas^2, \qquad m=|\Pzero|=M-n+1.
    \end{equation*}
We compare $\overline{A_1}$ with the expression obtained by replacing $\sigma^2(x_k)$ by $\sigma^2(y_k)$, while keeping the root projections and polynomial factors evaluated at $x$. Throughout this comparison, the weights remain fixed at the centre $y^i$ of $E_i$, as in \eqref{eq:weights}. To make the normalization explicit, write
    \begin{equation*}
        v_\alpha(z) :=\sum_{k=1}^N\sigma^2(z_k)(\alpha_k^*)^2 = \frac{\sum_{k=1}^N\sigma^2(z_k)\alpha_k^2}{\sum_{k=1}^N\sigma^2(y_k^i)\alpha_k^2}.
    \end{equation*}
In particular, $v_\alpha(y^i)=1$. Condition \ref{eq:sigma2:ratio}, applied to the neighbourhood centred at $y^i$, therefore gives
    \begin{equation*}
        |v_\alpha(z)-1|\leq\frac{\varepsilon_0}{8M}, \qquad z\in E_i,\quad \alpha\in\Pzero.
    \end{equation*}
Applying this bound separately to $z=x$ and $z=y$ yields
    \begin{equation*}
        |v_\alpha(x)-v_\alpha(y)| \leq |v_\alpha(x)-1|+|v_\alpha(y)-1| \leq\frac{\varepsilon_0}{4M}.
    \end{equation*}
Thus, no additional variance factor appears: the denominator in condition~\ref{eq:sigma2:ratio} is exactly $w_\alpha^2$, which is already included in $(\alpha_k^*)^2$. Since    
    \begin{equation*}
        \xas^2-\sum_{\substack{\beta\in\Pzero\\\beta\ne\alpha}}\xbs^2 =2\xas^2-Q,
    \end{equation*}
we have the exact decomposition
    \begin{equation*}
        \overline{A_1} = \left(\frac{H_+^w(x)}{\eA{n}}\right)^2 \sum_{\alpha\in\Pzero} v_\alpha(y)(2\xas^2-Q) + \left(\frac{H_+^w(x)}{\eA{n}}\right)^2 \sum_{\alpha\in\Pzero}\bigl(v_\alpha(x)-v_\alpha(y)\bigr)(2\xas^2-Q).
    \end{equation*}
The first summand is the frozen expression; in particular, it retains the terms involving $\sigma^2(y_k)$. Only the second summand is treated as an error. Its absolute value is bounded by
\begin{align*}
    \left(\frac{H_+^w(x)}{\eA{n}}\right)^2 \left| \sum_{\alpha\in\Pzero} \bigl(v_\alpha(x)-v_\alpha(y)\bigr)(2\xas^2-Q)\right|
    &\leq \frac{\varepsilon_0}{4M}\left(\frac{H_+^w(x)}{\eA{n}}\right)^2 \sum_{\alpha\in\Pzero}|2\xas^2-Q|\\
    &\leq \frac{\varepsilon_0(m+2)}{4M} \left(\frac{H_+^w(x)}{\eA{n}}\right)^2Q \leq \frac{\varepsilon_0}{2} \frac{H_+^w(x)}{\eA{n}} \leq-\frac14 A_5.
\end{align*}
Here, we used
    \begin{equation*}
        \sum_{\alpha\in\Pzero}|2\xas^2-Q| \leq\sum_{\alpha\in\Pzero}(2\xas^2+Q) =(m+2)Q,
    \end{equation*}
together with $2\leq m\leq M$ and $H_+^w(x)Q\leq\eA{n}$ from \eqref{eq:ineq_en_K}. For the final inequality, \ref{eq:k:lowerbound}, $H_+^w(x)\leq\eaA{n-1}$ for $\alpha\in\Pzero$, and the definition of $\varepsilon_0$ give
    \begin{equation*}
        -A_5 \geq \frac{2\varepsilon}{\eA{n}} \sum_{\alpha\in\Pzero}|\alpha^*|^2\eaA{n-1} \geq 2\varepsilon_0\frac{H_+^w(x)}{\eA{n}}.
    \end{equation*}
It follows that
    \begin{eqnarray*}
        \overline{A_1}+\frac{1}{4}A_5 &\leq& \left(\frac{H^w_+(x)}{\eA{n}}\right)^2 \sum_{k=1}^N \sigma^2(y_k)\SaPz(\alpha_k^*)^2   \bigg(\xas^2-\SbPzna\xbs^2\bigg)\\
        &=&\left(\frac{H^w_+(x)}{\eA{n}}\right)^2 \sum_{k=1}^N \sigma^2(y_k)  \bigg(\SaPz 2(\alpha_k^*)^2 \xas^2-\SabPz(\alpha_k^*)^2 \xbs^2\bigg)\\
        &=& \bigg(\frac{H^w_+(x)}{\eA{n}}\bigg)^2 \sum_{\alpha\in\Pzero} \xas^2 \sum_{k=1}^N \sigma^2(y_k)\bigg(2(\alpha_k^*)^2-\sum_{\beta\in\Pzero}(\beta_k^*)^2\bigg)\/.
    \end{eqnarray*}
    To get the last expression note that we just collect coefficients appearing with fixed $\xas^2$ and $k$. Every $\xas^2$ appears once with positive $(\alpha_k^*)^2$, but also with $-(\beta_k^*)^2$ as many times as many $\beta\in \Pzero$ different from $\alpha$ we can find.
    
Let us now focus on the double sum of $2\alpha^*_k\beta^*_k\xas\xbs$ in the upper bound for $\overline{A_2}$ and consider $\alpha, \beta\in\Pzero$ such that $\alpha\neq \beta$. These two roots may contribute to the sum if and only if $\alpha_k\beta_k\neq 0$ for some $k$. In this case it may still happen that $\ab=0$, but this is only if $\alpha=e_i\pm e_j$, $\beta=e_i\mp e_j$ and then we simply have $\alpha_i^*\beta_i^*=-\alpha_j^*\beta_j^*$, $y_i=y_j$, $w_\alpha=w_\beta$ and finally
\begin{equation*}
\sigma^2(y_i)\alpha_i^*\beta_i^*\xas\xbs+\sigma^2(y_j)\alpha_j^*\beta_j^*\xas\xbs =0\/,
\end{equation*}
which means that this kind of pairs $\{\alpha,\beta\}$ does not contribute to the considered sum and we can focus on $\alpha,\beta\in \Pzero$, $\alpha\neq \beta$ such that $\ab\neq 0$.
Fix $\alpha\in \Pzero$. For every $\beta\in \Pzero$ such that $\beta\neq \alpha$ and $\ab\neq 0$ there exists 
\begin{equation*}
\gamma=\gamma_{\alpha,\beta} = \pm\reflect{\alpha}\beta\in R_+\/,
\end{equation*}
i.e., $\gamma$ (up to the sign) is a reflection of $\beta$ in $\alpha$. It is easy to see that $|\beta|=|\gamma|$,  $\ac = \mp \ab \neq 0$ and recall \eqref{eq:ac:form1}, which gives
\begin{equation}
    \label{eq:inner}
    2\ab \xa\xb+	2\ac \xa\xc = \frac{4\xa^2}{|\alpha|^2}.
\end{equation}

At the point $y$, an active difference root implies equality of the corresponding coordinates. An active sum root implies that the coordinates are opposite; in type $B_N$ they are both zero, whereas in type $D_N$ their variances are equal by \ref{ass:variance:compatibility}. Since $\alpha$ and $\beta$ are nonorthogonal, their supports intersect. Hence the variances are equal on the union of their supports, which also contains the support of $\gamma$. Denote this common variance at $y$ by $\sigma^2$. The same equality within the union of supports holds separately at $y^i$. Its common value need not equal $\sigma^2$. Since the weights are defined at $y^i$, and $|\beta|=|\gamma|$, we nevertheless obtain
    \begin{equation*}
        w_\gamma=w_\beta =\frac{|\beta|}{|\alpha|}w_\alpha.
    \end{equation*}

Taking into consideration only selected $\beta,\gamma$ from the sum over all roots from $\Pzero$ not equal to $\alpha$, but having nonzero inner product with $\alpha$
\begin{equation*}
    \sum_{k=1}^N \sigma^2(y_k) \big(2\alpha_k^*\beta_k^* \xas\xbs+2\alpha_k^*\gamma_k^*\xas\xcs\big)
\end{equation*}
we can use \eqref{eq:inner} to write it as
\begin{equation*}
    \frac{2\sigma^2|\alpha|^2\big(\ab \xa\xb+	\ac \xa\xc\big)}{|\beta|^2 w_\alpha^4 } = \frac{4\sigma^2\xa^2}{|\beta|^2 w_\alpha^4} = \frac{4}{|\alpha|^2|\beta|^2 }\sum_{k=1}^N \sigma^2(y_k)(\alpha_k^*)^2 \xas^2.
\end{equation*}
In summary, we see that the expression $\overline{A_2}$ will produce additional positive elements $\sigma^2(y_k)(\alpha_k^*)^2\xas^2$ with coefficient $4/(|\alpha||\beta|)^2$, which appears every time we can find a pair $\{\beta,\gamma\}$ as above. Note that $4/(|\alpha||\beta|)^2=1$ if $|\alpha|=|\beta|=\sqrt{2}$ and $4/(|\alpha||\beta|)^2={2}$ {if either $|\alpha|=1$ and $|\beta|=\sqrt2$ or the opposite; note that $\ab=0$ if $|\alpha|=|\beta|=1$ and $\alpha\neq\beta$.}

Recall that the collision order $m=M-n+1$ indicates the number of elements in $\Pzero$, and consider $n_\alpha$ as the number of appearances of $\xas^2\sigma^2(y_k)(\alpha_k^*)^2$ from the procedure described above. For $\alpha = e_i\pm e_j$ there is at most one pair $\{\beta,\gamma\}$ such that $|\beta|=|\gamma|=1$, which is $\{e_i,e_j\}$ and that will produce the considered expression with coefficient $2$. The rest of the pairs will produce the expression with coefficients equal to $1$. We bound the number of such pairs roughly by excluding four roots $\alpha = e_i\pm e_j$, $e_i$, $e_j$ and $e_i\mp e_j$, then counting the remaining roots in $\Pzero$ and dividing this number by $2$, since we do not want to count pairs twice. This leads to the upper bound $(m-4)/2$ for the number of such pairs. In summary, we get $n_\alpha\leq 2+(m-4)/2\leq m-2$ as $m\geq 4$ in this case since we have $e_i\pm e_j, e_i, e_j, e_i\mp e_j \in \Pzero$. If there are no pairs of roots with length $1$, but there is at least one pair of length $\sqrt{2}$, then we have the upper bound $(m-1)/2$ for possible pairs and get $n_\alpha\leq (m-1)/2\leq m-2$, where the last inequality follows as $m\geq 3$ in this case. Finally, if $\alpha=e_i$, then every possible pair will produce the expression with coefficient $2$, but there are at most $(m-2)/2$ such pairs, because we have to exclude all other roots of length $1$ as they are orthogonal to $\alpha$. This also gives $n_\alpha\leq m-2$.

We can now go back to the bounds of $\overline{A_1}+\overline{A_2}+\frac{1}{2}A_5$ and get
\begin{eqnarray*}
        \overline{A_1}+\overline{A_2}+\frac{3}{8}A_5 &\leq& \left(\frac{H^w_+(x)}{\eA{n}}\right)^2\sum_{\alpha\in\Pzero} \xas^2 \sum_{k=1}^N \sigma^2(y_k)\left((\alpha_k^*)^2(2+n_\alpha)-\sum_{\beta\in\Pzero} (\beta_k^*)^2\right)\\
        &\stackrel{\eqref{eq:weights}}{\leq}& \left(\frac{H^w_+(x)}{\eA{n}}\right)^2\sum_{\alpha\in\Pzero} \xas^2 \left[m\frac{\sum_{k=1}^N \sigma^2(y_k)\alpha_k^2}{\sum_{k=1}^N \sigma^2(y_k^i)\alpha_k^2}-\sum_{\beta\in\Pzero}\frac{\sum_{k=1}^N \sigma^2(y_k)\beta_k^2}{\sum_{k=1}^N \sigma^2(y_k^i)\beta_k^2}\right]\\
        &\stackrel{\text{\ref{eq:sigma2:ratio}}}{\leq}&  \left(\frac{H^w_+(x)}{\eA{n}}\right)^2\sum_{\alpha\in\Pzero} \xas^2\left[m\left(1+\frac{\varepsilon_0}{8M}\right)-m\left(1-\frac{\varepsilon_0}{8M}\right)\right] \\
        &=&\frac mM\cdot \frac{\varepsilon}{4}\left(\frac{H^w_+(x)}{\eA{n}}\right)^2\sum_{\alpha\in\Pzero} \min_{\beta\in\Pzero}|\beta^*|^2\xas^2\\
        &\leq& \frac{\varepsilon}{4}\left(\frac{H^w_+(x)}{\eA{n}}\right)^2\sum_{\alpha\in\Pzero} |\alpha^*|^2\xas^2 \leq \frac{\varepsilon}{4} \frac{1}{\eA{n}}\sum_{\alpha\in\Pzero}|\alpha^*|^2\eaA{n-1}\/,
    \end{eqnarray*}
which directly gives, by using \ref{eq:k:lowerbound} and $A_5\leq 0$, that
    \begin{eqnarray*}
        \overline{A_1}+\overline{A_2}+\frac{3}{4}A_5 &\leq & \frac{1}{\eA{n}}\sum_{\alpha\in\Pzero}|\alpha^*|^2\eaA{n-1}\left(\frac{\varepsilon- 2k_\alpha(x)}{4} \right)\leq 0\/.
    \end{eqnarray*}

%%%%%%%%%%%%%%%%%%%%%%%%%%%%%%%%%%%%% STEP 2
{\it Step 2}. 
    In the next step, we show that the drift coefficients $b(x_k)dt$ appearing in \eqref{eq:x:SDE} cannot cause multiple collisions. More precisely, the drift part $A_5dt$ in the stochastic description of $dS_n^w$, which comes from the repulsive part prevents collisions that could be caused by the drift part $A_4dt$ described by the function $b$ in the following way 
    \begin{eqnarray*}
    A_4+\frac{1}{4} A_5 &\leq& \frac{2}{\eA{n}}\sum_{k=1}^N |b(x_k)|\SaR |\alpha_k^*|\xas\eaA{n-1}- \frac{1}{2\eA{n}}\sum_{\aR} |\alpha^*|^2\eaA{n-1}k_\alpha(x)\\
	 &\leq& \frac{2c_1}{\eA{n}}\bigg(\sum_{\alpha\in \Pplus}+\sum_{\alpha\in \Pzero}\bigg) \xas\,\eaA{n-1}- \frac{1}{2\eA{n}}\sum_{\aR} |\alpha^*|^2\eaA{n-1}k_\alpha(x)\\
	    &\leq&2c_1c_2\sum_{\alpha\in \Pplus} \frac{\xas^2\eaA{n-1}}{\eA{n}} + \sum_{\alpha\in \Pzero}\frac{[4c_1\xas-  |\alpha^*|^2k_\alpha(x)]\eaA{n-1}}{2\eA{n}}\/.
\end{eqnarray*} 
As we have previously observed, every component of the first sum is simply bounded by $1$. To deal with the other sum, it is enough to see that $4c_1\xas-|\alpha^*|^2 k_\alpha(x)$ becomes negative near $y$, because $k_\alpha(y)$ is strictly positive by \ref{ass:sigma:k:positive} and $\xas$ goes to $0$ as $\left<y,\alpha^*\right>=0$ for every $\alpha\in \Pzero$.

\bigskip

%%%%%%%%%%%%%%%%%%%%%%%%%%%%%%%%%%%%% STEP 3
{\it Step 3}. In the last step, we have to show that the two remaining parts of the drift $A_3$ and $A_6$ do not cause an explosion. Here, we do not need to compensate for those parts with the help of $A_5$. Indeed, taking into account every single component of the inner sum defining $A_3$ in \eqref{S:A3:form}, which is
\begin{equation}
    \label{eq:A22:elements}
    \alpha_k^* \beta_k^* \frac{2\xas\xbs\eabA{n}\eabA{n-2}}{(\eA{n})^2}\/,\qquad \alpha,\beta \in R_+\/,\alpha\neq \beta\/,
\end{equation}
we can consider all possible scenarios for $\alpha$ and $\beta$. If both of them are in $\Pzero$, then by \eqref{eq:ineq_en_K} we get
 \begin{equation*}
     \frac{2\xas\xbs}{\eA{n}} \leq \frac{\xas^2+\xbs^2}{\eA{n}} \leq c_3\frac{(\xas^2+\xbs^2)H_+^w(x)}{\eA{n}}\leq c_3\/, 
 \end{equation*}
 which together with the obvious bound $\eabA{n}\leq \eA{n}$ gives finiteness of \eqref{eq:A22:elements}. If at least one of the roots $\alpha$ or $\beta$ is in $\Pplus$, then all components of $\eabA{n}$ must contain a product of the form $\xcs^2\xds^2$ for two different roots $\gamma,\delta\in \Pzero$. Since, as we have seen just before, for every $\gamma\in \Pzero$ we have $\xcs^2 \leq c_3 \eA{n}$, this implies the finiteness of $\eabA{n}/(\eA{n})^2$ and consequently \eqref{eq:A22:elements} as well. 

    \medskip

 The final part relates to $A_6$, which is given by \eqref{S:A6:form} and can be slightly rewritten as follows
 \begin{equation*}
    A_6 = - 2\sum_{\bR} k_\beta(x)\SaRnb	\frac{\abs\xas}{\xbs} \frac{\eaA{n-1}}{\eA{n}}\/.
 \end{equation*}
 First suppose that $\beta\in\Pplus$. Then $k_\beta(x)/\xbs$ is bounded on the neighbourhood under consideration. Moreover,
    \begin{equation*}
        \frac{\xas\eaA{n-1}}{\eA{n}} \leq\frac{1}{\xas}.
    \end{equation*}
For $\alpha\in\Pplus$ the right-hand side is bounded. For $\alpha\in\Pzero$ it is locally integrable, by \ref{eq:k:lowerbound} and local integrability of the singular drift associated with $\alpha$. Thus, the contribution with $\beta\in\Pplus$ is absolutely integrable.
 
 Thus, from now on, we fix $\beta\in \Pzero$ and focus on the inner sum and take $\alpha\in R_+$ different from $\beta$, and we only consider those roots for which $\ab\neq 0$. Then, recall that there exists $\gamma = \gamma_{\beta,\alpha}\in R_+$ such that $\gamma=\pm\reflect{\beta}\alpha$.
 Moreover, we can write $\eaA{n-1} = \xbs^2\eabA{n-2}+\eabA{n-1}$, and then we symmetrize over $\alpha$ and $\gamma$ in $\eabA{n-1} = \xcs^2\eabcA{n-2}+\eabcA{n-1}$. Collecting the components for $\alpha$ and $\gamma$ together and applying the facts that the reflection of $\gamma$ in $\beta$ is again $\alpha$, and $\beta\in\Pzero$ implies $w_\alpha=w_\gamma$, we can reduce our consideration into three parts
 \begin{eqnarray*}
     A_6^{(1)} &=& -2\sum_{\beta\in \Pzero}\frac{k_\beta(x)}{\xbs}\SaRnb\abs \xas\xbs^2 \frac{\eabA{n-2}}{\eA{n}}\/,\\
     A_6^{(2)} &=&-2\sum_{\beta\in \Pzero}k_\beta(x) \sum_{(\alpha,\gamma)\in R_+(\beta)}\xas\xcs\frac{\abs\xcs+\bcs\xas}{\xbs} \frac{\eabcA{n-2}}{\eA{n}}\\
      &=& -2\sum_{\beta\in \Pzero}k_\beta(x) \sum_{(\alpha,\gamma)\in R_+(\beta)}\frac{\xas\xcs}{w_\alpha w_\gamma}\frac{\ab\xc+\bc\xa}{\xb} \frac{\eabcA{n-2}}{\eA{n}}\\
      &\stackrel{\eqref{eq:ac:form2}}{=}& -4\sum_{\beta\in \Pzero}k_\beta(x) \sum_{(\alpha,\gamma)\in R_+(\beta)}\frac{\xas\xcs}{w_\alpha w_\gamma}\frac{\ab\bc}{|\beta|^2} \frac{\eabcA{n-2}}{\eA{n}}\/,\\
     A_6^{(3)} &=&-2\sum_{\beta\in \Pzero}k_\beta(x) \sum_{(\alpha,\gamma)\in R_+(\beta)}\frac{\abs\xas+\bcs\xcs}{\xbs} \frac{\eabcA{n-1}}{\eA{n}}\\
     &=&-2\sum_{\beta\in \Pzero}k_\beta(x) \sum_{(\alpha,\gamma)\in R_+(\beta)}\frac{1}{w_\alpha w_\gamma}\frac{\ab\xa+\bc\xc}{\xb} \frac{\eabcA{n-1}}{\eA{n}}\\
     &\stackrel{\eqref{eq:ac:form1}}{=}&-4\sum_{\beta\in \Pzero}k_\beta(x) \sum_{(\alpha,\gamma)\in R_+(\beta)}\frac{1}{w_\alpha w_\gamma |\beta|^2}\frac{\eabcA{n-1}}{\eA{n}}\/,
 \end{eqnarray*}
 where the notation $R_+(\beta)$ is introduced in \eqref{eq:Rplus}.
Since $R_+(\beta)$ contains both $(\alpha,\gamma)$ and
$(\gamma,\alpha)$, the contribution to $A_6$ from
$\beta\in\Pzero$ equals
\begin{equation*}
    A_6^{(1)}+\frac12 A_6^{(2)}+\frac12 A_6^{(3)}.
\end{equation*}
 
 Here, as indicated, we have used formulas from Proposition \ref{prop:AP:alfagamma}. Starting from the simplest one, observe that the last expression is just non-positive. To deal with the final form of $A_6^{(2)}$, observe that since $\beta\in\Pzero$, then for $(\alpha,\gamma)\in R_+(\beta)$ either both are in $\Pzero$ or both are in $\Pplus$. In the first case, we simply have 
 \begin{equation*}
    \frac{2\xas\xcs}{\eA{n}}   \leq \frac{\xas^2+\xcs^2}{\eA{n}} 
 \end{equation*}
and the last ratio is bounded. If $\alpha,\gamma\in \Pplus$, then there are only $(n-3)$ roots left in $\Pplus$ and consequently, every component of $\eabcA{n-2}$ consists of at least one factor $\xds^2$, where $\delta\in\Pzero$, which, similarly to before, makes the ratio $\eabcA{n-2}/\eA{n}$ bounded. Finally, since $\xbs^2/\eA{n}$ is bounded and $k_\beta(x)/\xbs$ is integrable {due to the fact that \eqref{eq:x:SDE} has a solution and therefore all of its terms are integrable,} all components of $A_6^{(1)}$ are integrable and consequently cannot explode.

\medskip

In summary of all the considerations conducted, we have just shown that from $n=M-1$ to $n=1$, the event $\{\tau_{n}<\tau_{n-1}\}$ has probability $0$. Recall that $e_0\equiv1$, which finishes the proof as then $\tau_0=\lifetime$. 
\end{proof}

The remainder, which completes the proof of Theorem \ref{th:multiplecollisions}, concerns showing that the process enters the interior of the Weyl chamber immediately after starting.

\begin{proposition}\label{pr:enteringW}
    Assume that \ref{ass:sigma:b:continuous} and \ref{ass:sigma:k:positive} hold. For $x=(x_1,\ldots,x_N)$ being a solution to \eqref{eq:x:SDE} define
    \begin{equation*}
        \tau_{\C} = \inf\{t>0: x(t)\in \C\}\/.
    \end{equation*}
    Then for every $x(0) \in \Cc$ we have $\tau_{\C} = 0$ a.s.
\end{proposition}
\begin{proof}
    Let us denote by $M$ the number of roots in $R_+$ and consider the symmetric polynomials $\eAc{n}$ for $n=1,\ldots,M$. We have
    \begin{eqnarray*}
        d\eAc{n} &=& 2\sum_{k=1}^N \sum_{\aR} \alpha_k \xa \eaAc{n-1}\sigma(x_k)dB_k
            +2\sum_{k=1}^N \sum_{\aR} \alpha_k \xa \eaAc{n-1} b(x_k)dt\\
        &&  +2\sum_{\aR}   \eaAc{n-1}|\alpha|^2 k_\alpha(x)\,dt + 2\sum_{\beta\in R_+} \frac{k_\beta(x)}{\xb} \mathop{\sum_{\aR}}_{\alpha \neq \beta , \ab\neq 0}\ab \xa \eaAc{n-1} dt \\
        &&  +2 \sum_{k=1}^N \mathop{\sum_{\aR}\sum_{\bR}}_{\beta\neq \alpha}\alpha_k\beta_k\xa\xb\eabAc{n-2}\sigma^2(x_k)dt \\
        &&  + \sum_{k=1}^N \sum_{\aR}\alpha_k^2 \eaAc{n-1}\sigma^2(x_k)dt\/. 
    \end{eqnarray*}
    As previously, to remove the singularity in $k_\beta/\xb$ we use \eqref{eq:ac:form1} and \eqref{eq:ac:form2} to get for fixed $\beta\in R_+$ that 
    \begin{eqnarray*}
        2 \mathop{\sum_{\aR}}_{\alpha \neq \beta , \ab\neq 0}\ab \xa \eaAc{n-1} &=&  \sum_{(\alpha,\gamma)\in R_+(\beta)}(\ab\xa+\cb\xc)\eacAc{n-1}\\
        &&+\sum_{(\alpha,\gamma)\in R_+(\beta)}\xa\xc(\ab\xc+\cb\xa)\eacAc{n-2}\\
        &\underset{\eqref{eq:ac:form2}}{\overset{\eqref{eq:ac:form1}}{=}}& \frac{2\xb}{|\beta|^2}\sum_{(\alpha,\gamma)\in R_+(\beta)} (\eacAc{n-1}+\ab\bc\xa\xc \eacAc{n-2}).
    \end{eqnarray*}
    
    Assume now that one of the processes $\eA{n}$, with $n=1,\ldots, M$, stays at zero for some positive time interval with positive probability. This means that the drift part of $\eA{n}$ cancels on this time interval. On the other hand, all products of $n$ distinct root projections, where each of them appears only once becomes zero as well. In particular, we have $\xa \eaAc{n-1} = 0$ and $\xa\xb\eabAc{n-2} = 0$ for every $\alpha,\beta \in R_+$ and $\alpha\neq \beta$. The desingularized interaction term is then nonnegative, since for each $\beta$ its remaining contribution is
    \begin{equation*}
        \frac{2k_\beta(x)}{|\beta|^2} \sum_{(\alpha,\gamma)\in R_+(\beta)} \eacAc{n-1}\geq0.
    \end{equation*}
Consequently, positivity of $k_\alpha$ and $\sigma^2$ imply that for every $\alpha\in R_+$ we have
        \begin{equation*}
            \sum_{\aR}   \eaA{n-1} |\alpha|^2 \ = 0
        \end{equation*}
    on the positive time interval with positive probability, i.e., $\eA{n-1}=0$ and inductively 
        \begin{equation*}
            \eA{M} = \ldots = \eA{1} = \eA{0} = 0
        \end{equation*}
    Since $\eA{0}\equiv 1$ we get a contradiction. This means that the process enters immediately the interior of the Weyl chamber.
\end{proof}

To complete the proof of Theorem~\ref{th:multiplecollisions},
apply the preceding interior-start result to the shifted equation at each deterministic rational time $r>0$, on the event $\{r<T_\infty,\ x(r)\in\C\}$. The preceding argument also excludes any positive time interval spent entirely on the boundary. By continuity, every interval $(0,t)$ with $t<T_\infty$ therefore contains a rational time $r$ with $x(r)\in\C$. Taking the countable intersection of the corresponding probability-one events excludes simultaneous collisions at every $t\in(0,T_\infty)$, including for boundary initial conditions.

%%%%%%%%%%%%%%%%%%%%%%%%%%%%%%%%%%%%%%%%%%%%%%%%%%%%%%%%%%%%%%%%%%%%%%%%%%%%%%%%%%%%%%%%%%%%%%%%%%%%%%%%%%%%%%%%%%%%%%%%%%%%%%%%%%%%%%%%%%%%%%%%%%%%%%%%%%%%%%%%%%%%%%%%%%%%%%%%%%%%%%%%%%%%%%%%%%%%%%%%%%%%%%%%%%%%%%%%%%%%%%%%%%%%%%%%%%%%%%%%%%%%%%%%%%%%%%%%%%%%%%%%%%%%%%%%%%%%%%%%%%%%%%%%%%%%%%%%%%%%%%%%%%%%%%%%%%%

\section{Hausdorff dimension bounds}\label{sec:dimensionbounds}

In order to find bounds on the Hausdorff dimension of collision times, we adopt the following general strategy. First, we choose a functional of our general particle system, the squared projection of $x$ onto a simple root $\beta$, that cancels whenever there is a collision with the hyperplane defined by $\beta$; the union of zero sets of all squared projections onto simple roots then form the set of collision times. This functional has upper and lower bounds given by a time-transformed Bessel process. Then, by verifying that the transformation in each case is bi-Lipschitz, we can relate the Hausdorff dimensions of the functional and the time-transformed Bessel process; the local bi-Lipschitz property follows from continuity and strict positivity of the diffusion coefficients. Finally, we use the Hausdorff dimension of times where Bessel processes hit zero to collect our results. We note here that the Hausdorff dimension bounds given here are valid for every starting point in $\Cc$. Indeed, the only problematic starting points are those in $\Cb$, or boundary initial configurations; these are configurations for which at least one positive-root projection vanishes. By Proposition~\ref{pr:enteringW} we know that particles separate immediately, as the process leaves the boundary of $\C$ immediately after starting, and the contribution of such a starting point does not change the dimensionality of the collision time set.

Let us emphasize that for the calculations that follow, we reset the weights $w$ as
        \[w_\alpha:=|\alpha|\]
so that $\alpha^*=\alpha/|\alpha|$ and $|\alpha^*|=1$. A critical fact that we use in our derivations follows immediately from Theorem~\ref{th:multiplecollisions} and is given below.
\begin{corollary}\label{cor:projectiondistance}
Assume $|R_+|\geq2$. Almost surely, for every compact interval $[a,b]\subset(0,T_\infty)$,
\[
 d_{a,b}:=\min_{t\in[a,b]}
 \min_{\substack{\alpha,\gamma\in R_+\\\alpha\ne\gamma}}
 \bigl(\xtas+\xtcs\bigr)>0.
\]
Consequently, if $0<\varepsilon<d_{a,b}/2$ and
$\xtbs\leq\varepsilon$ for $t\in[a,b]$, then
\[
 \xtas\geq d_{a,b}-\varepsilon,
 \qquad \alpha\ne\beta.
\]
For $|R_+|=1$, conditions concerning other projections are omitted.
\end{corollary}
\begin{proof}
    The absence of simultaneous root-wall collisions (Theorem~\ref{th:multiplecollisions}) makes every sum in the displayed minimum strictly positive at each positive time before the lifetime. There are finitely many such continuous sums. Compactness gives the result.
\end{proof}

\subsection{Upper bound}\label{sec:upperbound}

    \begin{proof}[Proof of Theorem~\ref{th:Hausdorffupper}]
        Let us recall that a collision occurs if and only if there exists a root $\beta\in R_+$ such that $\xb=0$. Let us also recall that $\beta\leq\alpha$  implies $\xb\leq\xa$, and the minimal roots are all simple roots by \ref{prop:simpleroots:linearcombination}. So we can simply focus on the \emph{squared} projections onto a simple root $\beta$,
        \[y_\beta:=\xbs^2\]
        to study the collision times.
        The SDE of $y_\beta$ reads
        \begin{align}
            d y_\beta &= 2\xbs\inner{\beta^*}{dx}+\inner{\beta^*}{dx}^2\notag\\
            &= 2\xbs\Sni\beta_i^*\big(\sigma(x_i)\, dB_i+b(x_i)\, dt\big)+2\xbs\SaR k_\alpha(x)\frac{\inner{\alpha^*}{\beta^*}}{\xas}\, dt+\Sni(\beta_i^*)^2 \sigma^2(x_i)\, dt.\notag
        \end{align}
        The quadratic variation is
        \[4\xbs^2\Sni(\beta_i^*)^2\sigma^2(x_i)\, dt=4y_\beta\Sni(\beta_i^*)^2\sigma^2(x_i)\, dt,\]
        so by the Lévy characterization theorem, and introducing the new Wiener process $W$, we can rewrite the martingale part as
        \[2\sqrt{y_\beta}\sqrt{\Sni(\beta_i^*)^2\sigma^2(x_i)}\, dW.\]
        In addition, we have
        \begin{align*}
            2\xbs\SaR k_\alpha(x)\frac{\inner{\alpha^*}{\beta^*}}{\xas}=2k_\beta(x)+2\sqrt{y_\beta}\SaRnb k_\alpha(x)\frac{\inner{\alpha^*}{\beta^*}}{\sqrt{y_\alpha}},
        \end{align*}
        so the SDE becomes
        \begin{align*}
            dy_\beta &= 2\sqrt{y_\beta}\sqrt{\Sni (\beta_i^*)^2\sigma^2(x_i)}\, dW+2\sqrt{y_\beta}\Sni\beta_i^*b(x_i)\, dt+2k_\beta(x)\, dt\\
            &\quad+2\sqrt{y_\beta}\SaRnb k_\alpha(x)\inner{\alpha^*}{\beta^*}\frac{dt}{\sqrt{y_\alpha}}+\Sni(\beta_i^*)^2\sigma^2(x_i)\, dt.
        \end{align*}
        Now, we perform the following time change,
        \[s=\Theta_\beta(t):=\int_0^tC_\beta(\tau)\, d\tau,\quad C_\beta(t):= \Sni(\beta_i^*)^2\sigma^2(x_i(t))>0.\]
        The time transformation $\Theta_\beta$ is bi-Lipschitz on every compact time interval before $T_\infty$, because $C_\beta$ is continuous and strictly positive. Its range is
        \[
        [0,S_\beta),\qquad
        S_\beta:=\Theta_\beta(T_\infty-)
        =\lim_{t\uparrow T_\infty}\int_0^t C_\beta(u)\,du.
        \]
        All coefficients in the following equations in $s$-time are evaluated at $x(\Theta_\beta^{-1}(s))$, and the Brownian motion in $s$-time is the corresponding time-changed Brownian motion.
        The SDE in $s$-time reads
        \begin{align}
            dy_\beta &= 2\sqrt{y_\beta}\, dW+2\frac{\sqrt{y_\beta}}{C_\beta}\Sni\beta_i^*b(x_i)\, ds+2\frac{k_\beta(x)}{C_\beta}\, ds+2\sqrt{y_\beta}\SaRnb \frac{k_\alpha(x)}{C_\beta}\inner{\alpha^*}{\beta^*}\frac{ds}{\sqrt{y_\alpha}}+\, ds\notag\\
            &= 2\sqrt{y_\beta}\, dW+2\frac{\sqrt{y_\beta}}{C_\beta}\bigg(\Sni\beta_i^*b(x_i)+\SaRnb k_\alpha(x)\inner{\alpha^*}{\beta^*}\frac{1}{\sqrt{y_\alpha}}\bigg)\, ds +\bigg(2\frac{k_\beta(x)}{C_\beta}+1\bigg)\,ds.\label{eq:sqprojtobound}
        \end{align}
        Our task now is to find an appropriate lower bound for this process. Let us define
        \[\check{\eta}_\beta:=\inf_{y\in W}\frac{k_\beta(y)}{\Sni(\beta_i^*)^2\sigma^2(y_i)}\]
        in order to bound the second drift term as follows,
        \[\bigg(2\frac{k_\beta(x)}{C_\beta}+1\bigg)\geq(2\check\eta_\beta+1).\]
        Next, we introduce the constants
        \begin{align}
            \hat{b}:=\sup_{y\in D}\left|\frac{b(y)}{\sigma^2(y)}\right|,\label{eq:bhat}
        \end{align}
        and
	\begin{align}
			c_R&:=\max_{\substack{\alpha\in R_+\\i\in\{1,\ldots,N\}:\\\alpha_i\neq 0}}\left|\frac{1}{\alpha^*_i}\right|.\label{eq:constcr}
	\end{align}
        which allow us to write
        \[\frac{1}{C_\beta}\Sni\beta_i^*b(x_i)=\frac{1}{C_\beta}\sum_{\substack{1\leq i\leq N\\\beta_i\ne0}}(\beta_i^*)^2\sigma^2(x_i)\frac{1}{\beta_i^*}\frac{b(x_i)}{\sigma^2(x_i)}\geq -\frac{c_R\hat{b}}{C_\beta}\Sni(\beta_i^*)^2\sigma^2(x_i)=-c_R\hat{b}.\]
        We also define
        \begin{align}
            \tilde{\eta}_\alpha:=\max_{1\leq i\leq N}\sup_{y\in W}\frac{k_\alpha(y)}{\sigma^2(y_i)}\label{eq:etatilde}
        \end{align}
        in order to obtain
        \begin{align*}
            \frac{1}{C_\beta}\SaRnb k_\alpha(x)\inner{\alpha^*}{\beta^*}\frac{1}{\sqrt{y_\alpha}}&=\frac1{C_\beta}\SaRnb \Sni(\beta_i^*)^2k_\alpha(x)\inner{\alpha^*}{\beta^*}\frac{1}{\sqrt{y_\alpha}}\\
            &\geq -\frac1{C_\beta}\SaRnb \tilde{\eta}_\alpha\Sni(\beta_i^*)^2\sigma^2(x_i)\frac{1}{\sqrt{y_\alpha}}=-\SaRnb \frac{\tilde{\eta}_\alpha}{\sqrt{y_\alpha}},
        \end{align*}   
        which is obtained by recalling that $|\beta^*|^2=\Sni(\beta_i^*)^2=1$ and that $\abs\geq-1$ for the root systems we consider.
        With these inequalities, we can write the bound
        \begin{align}
            dy_\beta\geq2\sqrt{y_\beta}\, dW-2\sqrt{y_\beta}\bigg(c_R \hat b+\SaRnb \frac{\tilde\eta_\alpha}{\sqrt{y_\alpha}}\bigg)\, ds +(2\check\eta_\beta+1)\,ds.\label{eq:sdenewlowerbound}
        \end{align}
        Set
\begin{equation}
 \widetilde h:=\max_{\alpha\in R_+}\widetilde\eta_\alpha,
 \label{eq:htilde}
\end{equation}
and, for $c>0$, put
\[
 K_c:=c_R\hat b+\frac{(M-1)\widetilde h}{c}.
\]
Fix a simple root $\beta$. For $L,c,\varepsilon>0$, consider the relatively open set
\[
\begin{split}
 U_{\beta,L,c,\varepsilon}:=\{x\in\overline W:
 &|x|<L,\quad \xbs<\varepsilon,\\
 &\xas>c\text{ for every }\alpha\ne\beta\}.
\end{split}
\]
On this set, the preceding drift bound gives
\[
 dy_\beta\geq2\sqrt{y_\beta}\,dW+
 d_{\beta,c,\varepsilon}\,ds,
 \qquad
 d_{\beta,c,\varepsilon}:=1+2\check\eta_\beta-2\varepsilon K_c.
\]
Choose $\varepsilon$ sufficiently small that $d_{\beta,c,\varepsilon}>0$.

At each deterministic rational time $r>0$, on the event
$\{r<T_\infty,\ x(r)\in U_{\beta,L,c,\varepsilon}\}$, stop at the first exit from this set or at $r+1$, whichever occurs first, and reset the $s$-time at $r$. Up to this stopping time, comparison with a squared Bessel process $Z$ of dimension $d_{\beta,c,\varepsilon}$, started from zero and driven by the same Brownian motion in $s$-time, gives $Z\leq y_\beta$. The zero set of $y_\beta$ on this interval is therefore contained in the zero set of $Z$. Its Hausdorff dimension is at most
\[
 \max\{0,\tfrac12-\check\eta_\beta+\varepsilon K_c\}.
\]
The same bound holds in $t$-time by local bi-Lipschitz continuity. Brownian motions and comparison processes may be continued after the stopping time on an enlarged space; only their restrictions before the stopping time are used.

For fixed $L\in\N$ and $c\in\mathbb Q_{>0}$, define
\[
\begin{split}
 F_{\beta,L,c}:=\{t\in(0,T_\infty):
 &\xtbs =0,\quad |x(t)|<L,\\
 &\xtas>c\text{ for every }\alpha\ne\beta\}.
\end{split}
\]
For each fixed admissible $\varepsilon$, every time in $F_{\beta,L,c}$ is covered by one of the stopped intervals just constructed. Indeed, continuity provides a preceding rational time $r$ with $t-r<1$ such that the trajectory remains in $U_{\beta,L,c,\varepsilon}$ from $r$ through $t$.

Take a countable sequence of admissible $\varepsilon$ decreasing to zero. Countable stability, first over the rational starting times and then the limit in $\varepsilon$, gives
\[
 \dim F_{\beta,L,c}\leq
 \max\{0,\tfrac12-\check\eta_\beta\}.
\]
At a positive-time hit of the $\beta$-wall, every other positive-root projection is strictly positive by Theorem~\eqref{th:multiplecollisions}. Thus the positive-time zero set of $y_\beta$ is the countable union of $F_{\beta,L,c}$ over $L\in\N$ and $c\in\mathbb Q_{>0}$. Hence
\[
 \dim y_\beta^{-1}(0)\leq
 \max\{0,\tfrac12-\check\eta_\beta\}.
\]
When $M=1$, omit all restrictions on other projections and set $K_c=c_R\hat b$.
        We finish by using \ref{prop:Hausdorff:countablestability} and writing
        \begin{align*}
          \dim x^{-1}(\Cb)&= \dim \bigcup_{\beta\in\SimplePlus}y_\beta^{-1}(0)=\max_{\beta\in\SimplePlus}\dim y_\beta^{-1}(0)\leq \max_{\beta\in\SimplePlus} \max\bigg\{0,\frac12-\check{\eta}_\beta\bigg\}=\max\bigg\{0,\frac12-\min_{\beta\in\SimplePlus}\check{\eta}_\beta\bigg\}.
        \end{align*}
    \end{proof}
    The upper-bound proof also applies when the coefficient ratios used to bound the regular drift and the other-root interactions are only locally bounded. On a compact localization replace $\hat b$ and $\widetilde\eta_\alpha$ by their local bounds. For each fixed compact set, let $\varepsilon\downarrow0$ before taking a countable compact exhaustion. The same global lower bound $\check\eta_\beta$ for the self-interaction ratio may be retained throughout. This also applies on an open coefficient domain, up to its exit time, provided that the continuity, positivity, variance compatibility, and absence of simultaneous root-wall collisions hold on the localized domains.
\subsection{Lower bound}\label{sec:lowerbound}
    \begin{proof}[Proof of Theorem~\ref{th:Hausdorfflower}]
        In the same vein as the proof of 
        Theorem~\ref{th:Hausdorffupper}, we consider the squared projection onto the arbitrary simple root $\beta$, and carry on the same calculations up to \eqref{eq:sqprojtobound}.
        By \ref{ass:k:monotonicty}, we can show that 
        \[2\SaRnb k_\alpha(x)\inner{\alpha^*}{\beta^*}\frac{1}{\sqrt{y_\alpha}}\]
        is not positive. We begin by expanding the sum as follows,
        \begin{align*}
            2\SaRnb k_\alpha(x)\inner{\alpha^*}{\beta^*}\frac{1}{\sqrt{y_\alpha}}&=\SaRnb \frac{k_\alpha(x)}{\xas}\inner{\alpha^*}{\beta^*}+\sum_{\reflect{\beta}\alpha\in R_+\setminus \{\beta\}} \frac{k_{\reflect{\beta}\alpha}(x)}{\inner{x}{\reflect{\beta}\alpha^*}}\inner{\reflect{\beta}\alpha^*}{\beta^*}.
        \end{align*}
        The rhs follows from the fact that for any simple root $\beta$, every positive root $\alpha\neq\beta$ is reflected onto another positive root, namely $\reflect{\beta}\alpha\in R_+\setminus\{\beta\}$. 
        We can rewrite this expression using \eqref{eq:Rplus} and setting $\gamma=\reflect{\beta}\alpha$,
        \begin{align*}
            2\SaRnb k_\alpha(x)\inner{\alpha^*}{\beta^*}\frac{1}{\sqrt{y_\alpha}}&=\sum_{(\alpha,\gamma)\in R_+(\beta)}\Big( \frac{k_\alpha(x)}{\xas}\inner{\alpha^*}{\beta^*}+ \frac{k_{\gamma}(x)}{\inner{x}{\gamma^*}}\inner{\gamma^*}{\beta^*}\Big),
        \end{align*}
        while keeping in mind that $\inner{\gamma^*}{\beta^*}=\inner{\reflect{\beta}\alpha^*}{\beta^*}=\inner{\alpha^*}{\reflect{\beta}\beta^*}=-\inner{\alpha^*}{\beta^*}$.
        Now, we apply \ref{ass:k:monotonicty}. Fix a simple root $\beta$ and let $\gamma=\reflect{\beta}\alpha$, where $\alpha\in R_+\setminus\{\beta\}$. Recall that $\gamma\in R_+$ and $|\gamma|=|\alpha|$. Suppose first that $\inner{\alpha}{\beta}>0$. Then
            \begin{equation*}
                \alpha-\gamma =\frac{2\inner{\alpha}{\beta}}{|\beta|^2}\,\beta.
            \end{equation*}
        The coefficient on the right is a positive integer. Consequently, $\gamma\leq\alpha$ in the root order. Since $\alpha$ and $\gamma$ have the same length, assumption \ref{ass:k:monotonicty} gives, for $x\in\C$,
            \begin{equation*}
                \frac{k_\gamma(x)}{\inner{x}{\gamma}} \geq \frac{k_\alpha(x)}{\inner{x}{\alpha}}.
            \end{equation*}
        Using $\inner{\gamma}{\beta}=-\inner{\alpha}{\beta}$, we obtain
            \begin{equation*}
                \frac{k_\alpha(x)}{\inner{x}{\alpha^*}} \inner{\alpha^*}{\beta^*} +\frac{k_\gamma(x)}{\inner{x}{\gamma^*}} \inner{\gamma^*}{\beta^*} 
                     = \frac{\inner{\alpha}{\beta}}{|\beta|} \left(\frac{k_\alpha(x)}{\inner{x}{\alpha}}-\frac{k_\gamma(x)}{\inner{x}{\gamma}}\right) \leq 0.
            \end{equation*}
        If $\inner{\alpha}{\beta}<0$, the same conclusion follows by interchanging $\alpha$ and $\gamma$. Roots orthogonal to $\beta$ contribute zero. Thus, every reflected pair contributes a nonpositive quantity, proving that
            \begin{equation*}
                2\SaRnb k_\alpha(x)\frac{\inner{\alpha^*}{\beta^*}}{\sqrt{y_\alpha}}\leq 0.
            \end{equation*}
        The drift assumption directly gives
        \[
        \sum_{i=1}^N\beta_i^*b(x_i)
        =\frac1{|\beta|}\sum_{i=1}^N\beta_i b(x_i)\leq0
        \]
        for every simple root $\beta$. Therefore, for all simple roots $\beta$ the upper bound
        \begin{align*}
            dy_\beta &\leq 2\sqrt{y_\beta}\, dW+\bigg(2\frac{k_\beta(x)}{C_\beta}+1\bigg)\,ds
        \end{align*}
        is satisfied for the root systems in consideration in $s$-time.
        
        We now define for every simple root
        \begin{align}
            \hat\eta_\beta:=\sup_{y\in W}\frac{k_\beta(y)}{\Sni(\beta_i^*)^2\sigma^2(y_i)}=\sup_{y\in W}\frac{k_\beta(y)}{C_\beta(y)}\label{eq:etahat}
        \end{align}
        in order to write
        \begin{align*}
            dy_\beta &\leq 2\sqrt{y_\beta}\, dW+\bigg(2\hat{\eta}_\beta+1\bigg)\,ds.
        \end{align*}
        This inequality implies that $y_\beta$ is bounded above by a squared Bessel process, say $\hat z_{\hat{\eta}_\beta}$, of dimension $2\hat{\eta}_\beta+1$ which hits zero almost surely when $\hat{\eta}_\beta<1/2$ and has a Hausdorff dimension of hitting times at zero given almost surely by \cite{GoeingJaeshckeYor03,LiuXiao98}
        \[\dim \hat z_{\hat{\eta}_\beta}^{-1}(0) = \max\Big\{0,\frac12-\hat{\eta}_\beta\Big\}.\]
        
        We must ensure that $s$-time runs to infinity as $t\uparrow T_\infty$ in order to use the full zero set of $\hat z_{\hat{\eta}_\beta}$. If there exists a compact set $\Phi(\omega)\subset\Cc$ such that $x(t)\in\Phi(\omega)$ for any $t$ greater than some stopping time $\tau$, then we can always find a lower bound for $\sigma$ within $\Phi$, say $\sigma_\Phi$. Non-explosion of $x(t)$ follows from the trajectory bound assumed in the theorem. Indeed, on the event $\{T_\infty<\infty\}$, this bound would give
            \begin{equation*}
                \sup_{0\leq t<T_\infty}|x(t)| \leq \sqrt{N}\,K(1+T_\infty)^r<\infty,
            \end{equation*}
        contradicting the definition of $T_\infty$ as the limit of the exit times from balls of increasing radius. Consequently, $T_\infty=\infty$ almost surely. Then, the time transformation $\Theta_\beta$ maps $[0,\infty)$ to itself, as we have
        \[\lim_{t\to\infty}\int_0^tC_\beta(u)\, du\geq\lim_{t\to\infty}\int_\tau^t\Sni(\beta_i^*)^2\sigma^2(x_i(u))\, du\geq \lim_{t\to\infty}\int_\tau^t\sigma^2_\Phi\Sni(\beta_i^*)^2\, du=\lim_{t\to\infty}\sigma^2_\Phi\int_\tau^t du=\infty.\]
        In this case, we do not need to impose the decay condition on $\sigma$ and we can use the full zero set of $\hat z_{\hat \eta_\beta}$. Otherwise, we assume that $x(t)$ and $\sigma$ satisfy the bounds in the statement. Then, we will eventually, namely for large enough $t$, have the following estimate,
        \begin{align*}
            C_\beta(x(t))=&\Sni(\beta_i^*)^2\sigma^2(x_i(t))\geq \Sni(\beta_i^*)^2a^2(1+|x_i(t)|)^{-2p}\\
            &\geq \Sni(\beta_i^*)^2a^2(1+K(\omega))^{-2p}(1+t)^{-2pr}=a^2(1+K(\omega))^{-2p}(1+t)^{-2pr},
        \end{align*}
        and the integral of this object from $0$ to $\infty$ is divergent due to the requirement that $2pr\leq1$.
        
        Now that we have established that the time horizon extends to infinity, we proceed to use the zero set of $\hat z_{\hat{\eta}_\beta}$ to bound the Hausdorff dimension from below. Put $h_*:=\min_{\beta\in\Delta_+}\hat\eta_\beta$. If $h_*\geq1/2$, the asserted bound is zero and needs no hitting argument. Otherwise choose a minimizing root $\beta_*$ and recall that $\hat z_{h_*}^{-1}(0)\subseteq y_{\beta_*}^{-1}(0)$. Now, we write
        \begin{align*}
            \dim\Big(x^{-1}(\Cb)\Big)&=\dim\Big(\bigcup_{\beta\in\SimplePlus}y_\beta^{-1}(0)\Big)=\max_{\beta\in\SimplePlus}\dim y_\beta^{-1}(0)\\
            &\geq\dim y_{\beta_*}^{-1}(0)\geq\dim \hat z_{h_*}^{-1}(0)=\max\Big\{0,\frac12-h_*\Big\}.
        \end{align*}
        The second equality follows from \ref{prop:Hausdorff:countablestability}, and the inequality in the second line follows from \ref{prop:Hausdorff:monotonicity}. 
        \end{proof}

    \begin{proof}[Proof of Lemma~\ref{th:Hausdorfflowerweak}]
        In the same way as in the proof of Theorem~\ref{th:Hausdorfflower}, we can start from \eqref{eq:sqprojtobound}, as no assumptions are used up to that point. The next step is to try and bound the SDE of $y_\beta$ from above, and in order to do this, we impose conditions on $k_\alpha$, $\sigma$, and $b$. 
        Recall the constants $\hat b$, $c_R$, and $\tilde\eta_\alpha$ introduced in \eqref{eq:bhat}, \eqref{eq:constcr}, and \eqref{eq:etatilde}. With these, we write
        \[\frac{1}{C_\beta}\Sni\beta_i^*b(x_i)=\frac{1}{C_\beta}\sum_{\substack{1\leq i\leq N\\\beta_i\ne0}}(\beta_i^*)^2\sigma^2(x_i)\frac{1}{\beta_i^*}\frac{b(x_i)}{\sigma^2(x_i)}\leq \frac{c_R\hat{b}}{C_\beta}\Sni(\beta_i^*)^2\sigma^2(x_i)=c_R\hat{b},\]
        and
        \begin{align*}
            \frac{1}{C_\beta}\SaRnb k_\alpha(x)\inner{\alpha^*}{\beta^*}\frac{1}{\sqrt{y_\alpha}}&=\frac1{C_\beta}\SaRnb \Sni(\beta_i^*)^2k_\alpha(x)\inner{\alpha^*}{\beta^*}\frac{1}{\sqrt{y_\alpha}}\\
            %&=\frac1{C_\beta}\SaRnb \Sni(\beta_i^*)^2\frac{k_\alpha(x)}{\sigma^2(x_i)}\sigma^2(x_i)\inner{\alpha^*}{\beta^*}\frac{1}{\sqrt{y_\alpha}}\\
            &\leq \frac1{C_\beta}\SaRnb \tilde{\eta}_\alpha\Sni(\beta_i^*)^2\sigma^2(x_i)\frac{1}{\sqrt{y_\alpha}}=\SaRnb \frac{\tilde{\eta}_\alpha}{\sqrt{y_\alpha}}.
        \end{align*}
        We carried out these calculations in similar way to those in the proof of Theorem~\ref{th:Hausdorffupper}.
        We can now make use of $\hat\eta_\beta$, as defined in \eqref{eq:etahat} to write
        \begin{align}
            dy_\beta &= 2\sqrt{y_\beta}\, dW+2\sqrt{y_\beta}\bigg(\frac{1}{C_\beta}\Sni\beta_i^*b(x_i)+\SaRnb \frac{k_\alpha(x)}{C_\beta}\inner{\alpha^*}{\beta^*}\frac{1}{\sqrt{y_\alpha}}\bigg)\, ds +\bigg(2\frac{k_\beta(x)}{C_\beta}+1\bigg)\,ds\label{eq:ybeta:stime}\\
            &\leq 2\sqrt{y_\beta}\, dW+2\sqrt{y_\beta}\bigg(c_R\hat{b}+\SaRnb \frac{\tilde{\eta}_\alpha}{\sqrt{y_\alpha}}\bigg)\, ds +\Big(2\hat{\eta}_\beta+1\Big)\,ds.\notag
        \end{align}
        Put \[
 H:=\max_{\alpha\in\Delta_+}\hat\eta_\alpha.
\]
If $H\geq1/2$, the asserted bound is zero and the conclusion is immediate. Assume $H<1/2$.

We first prove a positive probability of hitting a simple wall before the lifetime. Fix a simple root $\beta$ and a point $z$ in the relative interior of its face. Set
\[
 q(x):=\xbs,\qquad
 P:=I-\beta^*(\beta^*)^{\mathsf T},
\]
where $I$ denotes the $N\times N$ identity matrix, and $(\beta^*)^{\mathsf T}$ denotes the transpose of $\beta^*$ taken as a column vector. Choose $r,h>0$ so that the closure of
\[
 U:=\{x\in\overline W:0\leq q(x)<h,\ |P(x-z)|<r\}
\]
is compact and meets no other root wall. Denote the first exit time from $U$ by $\tau_U$. On this closure choose constants $0<c_0\leq c_1<\infty$ with $c_0\leq C_\beta\leq c_1$. The part in parentheses of the second r.h.s. term in \eqref{eq:ybeta:stime} has absolute value at most a constant $K\geq0$. Thus, until exit from $U$, in $s$-time,
\[
 dy_\beta\leq2\sqrt{y_\beta}\,dW+
 \bigl(1+2H+2K\sqrt{y_\beta}\bigr)\,ds.
\]
After decreasing $h$, retaining bounds obtained on the original neighbourhood, assume
\[
 d:=1+2H+2Kh<2.
\]
For starting points satisfying $|P(x-z)|<r/4$ and $0<q(x)\leq q_0<h$, comparison up to exit from $U$ gives $y_\beta\leq Z$, where $Z$ is a squared Bessel process of dimension $d$ started from $q(x)^2$.

The tangential drift and diffusion coefficients, namely the parts of \eqref{eq:x:SDE:mult} orthogonal to $\beta^*$, are bounded on $\overline U$: the singular $\beta$-term vanishes after multiplication by $P$. Doob's martingale maximal inequality therefore allows us to choose $T>0$ small enough that
\[
 \pr \left(\sup_{t\leq T}|P(x(t\wedge\tau_U)-z)|\geq r\right)<\frac14,
\]
uniformly over these starting points. Decreasing $T$ and then choosing $q_0$ small enough also gives
\[
 \pr\left(\sup_{s\leq c_1T}Z_s\geq h^2\right)
 \leq\frac{q_0^2+dc_1T}{h^2}<\frac14.
\]
Here, $Z$ can be continued past the stopped time for the bound. Write $T_0^Z$ for its first zero. Since $d<2$, the squared-Bessel hitting-time scaling gives, after further decreasing $q_0$,
\[
 \pr(T_0^Z>c_0T)<\frac14.
\]
On the complement of these three exceptional events, $x$ cannot exit $U$ before the Bessel zero: tangential exit is excluded, and normal exit would imply that $Z$ has reached $h^2$. The $s$-time reaches $T_0^Z$ by $t$-time $T$, and comparison forces a hit of the $\beta$-wall. The probability of such a hit before exit from $U$ is therefore at least $1/4$.

The thin interior region just used is accessible with positive probability from any interior starting point. To see this, choose a smooth path from the starting point to a point strictly inside that region, and a compact tube around the path contained in $W$. The full drift is bounded and the diffusion matrix has a bounded inverse on this tube because $\sigma>0$. A bounded stopped Girsanov change of measure replaces the full drift by the velocity of the path. Under that measure, the difference between the stopped trajectory and the path is a continuous martingale with bounded diffusion coefficient. Traversing the path in a sufficiently short time, the martingale maximal inequality gives positive probability of staying in the tube and ending in the target region. Equivalence of the measures gives positive probability under the original measure. The same argument is uniform over a sufficiently small ball of initial points.

For a boundary initial point, immediate entrance by Proposition~\ref{pr:enteringW} and continuity imply that, at some deterministic positive time before the lifetime, the process lies in an interior ball with positive probability. Apply the preceding arguments conditionally at that time. More generally, the bounds above apply conditionally at deterministic times; no Markov assumption on the chosen weak solution is required. We have therefore established positive probability of a stopping time $\tau<T_\infty$ at which $q(x(\tau))=0$ and $x(\tau)\in U$.

On this event restart the $s$-time at $\tau$. For sufficiently small $0<\varepsilon<h$, stop when $q(x)$ reaches $\varepsilon$, when $|P(x-z)|$ reaches $r$, or at $\tau+1$. This interval has strictly positive length. Comparison up to its endpoint with a squared Bessel process started from zero and having dimension
\[
 d_\varepsilon:=1+2H+2K\varepsilon<2
\]
gives an upper bound for $y_\beta$. The zero set of this Bessel process is contained in the collision set. It has Hausdorff dimension $1-d_\varepsilon/2$ on every initial interval of positive length, hence also on the positive random interval under consideration. The local bi-Lipschitz time change gives
\[
 \dim x^{-1}(\partial W)\geq\tfrac12-H-K\varepsilon
\]
on the positive-probability event just established. Taking a countable sequence $\varepsilon\downarrow0$ proves the assertion.
    \end{proof}
        
\section{Particular cases}\label{sec:examples}

    Our results are readily applied to several well-known cases.
    \subsection{Multivariate Bessel processes}

    Also known as radial Dunkl processes \cite{RoslerVoit98,GallardoYor05}, these are processes given by the SDE
    \[dx_i=dB_i+\SaR k_\alpha\frac{\alpha_i}{\xa}\,dt,\]
    that is, $\sigma=1$, $b=0$, and $k_\alpha(x)=k_\alpha$, with the added constraint that $k_\alpha$ must be invariant under reflections along the roots in $R_+$, namely $k_{\reflect{\beta}\alpha}=k_\alpha$ for every $\alpha,\beta\in R_+$. Assuming that $R=A_{N-1}$, $B_N$, or $D_N$, from Theorems~\ref{th:multiplecollisions}-\ref{th:Hausdorfflower} the following is immediate.
    \begin{corollary}
        The collision times of the Bessel process with its Weyl chamber have the following Hausdorff dimension,
        \[\dim x^{-1}(\Cb)=\max\Big\{0,\frac12-\min_{\alpha\in\SimplePlus}k_\alpha\Big\}.\]
        \label{cor:multiBessel}
    \end{corollary}
    For variable coefficients satisfying \ref{ass:sigma:k:positive}-\ref{ass:k:monotonicty}, the corresponding upper and lower bounds apply; an exact Hausdorff dimension follows when the bounds coincide and the almost-sure lower-bound argument is justified. The requirement that $b$ be a decreasing function has the following physical interpretation: $b$ represents a force due to an external potential, namely $b(x_i)=-V'(x_i)$, and if $b$ is non-increasing then $V(x_i)$ is a convex function, and does not hinder particle collisions. For $\sigma=1$, constant root couplings, and $b(u)=-\lambda u$, $\lambda\geq0$, the drift of $|x|^2$ is
    \[
        N+2\sum_{\alpha\in R_+}k_\alpha-2\lambda|x|^2.
    \]
    The usual stopped moment bound gives non-explosion. We obtain $s=\Theta_\beta(t)=t$, so the lower comparison has an infinite time horizon. For $\lambda>0$, the upper bound is applied locally, as explained at the end of Section~\ref{sec:upperbound}. These processes are specialized for the several well-known multiple-particle systems, which we summarize below. In the remainder of the paper we assume $N\geq2$.
    \subsection{Dyson model}
    The Dyson model \cite{Dyson62A} is obtained by setting $R=A_{N-1}$, $\sigma(x)=1$, $b(x)=0$, and $k_\alpha(x)=k>0$, with the positive roots being $\{e_j- e_i\}_{1\leq i<j\leq N}$. The Weyl chamber is described by $ \overline W_{A_{N-1}}=\{x\in\R^N:x_i\leq x_j, 1\leq i<j\leq N\}$, and the process is given by the following SDE,
        \[dx_i=dB_i+k\sum_{j=1:j\neq i}^N\frac{dt}{x_i-x_j}.\]
        By Corollary~\ref{cor:multiBessel}, we recover the following familiar result \cite{HufnagelAndraus24}: the set of times where particle collisions occur in the Dyson model has a Hausdorff dimension given by
        \[\dim x^{-1}(\partial W)=\max\Big\{0,\frac12-k\Big\}.\]
    For $k=1/2$ and $k=1$, respectively, this process is the eigenvalue process of a suitably normalized real symmetric or complex Hermitian matrix Brownian motion \cite{KatoriTanemura04}. For these matrix parameters, positive-time particle collisions are known not to occur. The zero value of the dimension formula is consistent with that fact.

    Note in addition that there is a large freedom in this case, as by Theorems~\ref{th:multiplecollisions}--\ref{th:Hausdorfflower} we can add a non-increasing drift function $b$ without changing the Hausdorff dimension. In fact, the global bound on $b/\sigma^2$ is unnecessary for this application: the upper bound is localized, while the lower comparison uses the sign of the projected drift. Moreover, for every $u\in\R$, non-increase of $b$ gives
    \[
         u b(u)\leq u b(0).
    \]
    Consequently, the drift of $|x|^2$ is bounded above by a constant times $1+|x|^2$. A stopped moment bound gives non-explosion, and $\sigma=1$ makes it so that the time transformations $\Theta_\beta$ have an infinite horizon for every root $\beta$.

    \subsection{Multivariate Bessel process of type B}
    Of particular interest is the case where $R=B_N$, in which $k_\alpha$ is reduced to two independent parameters, $k_1$ for the positive roots $\{e_i\}_{1\leq i\leq N}$, and $k_2$ for the positive roots $\{e_j\pm e_i\}_{1\leq i<j\leq N}$. Its Weyl chamber is given by $\Cc _{B_N}=\{x\in\R^N:0\leq x_i\leq x_j, 1\leq i<j\leq N\}$, and its SDE reads
    \begin{align*}
        dx_i&=dB_i+\frac{k_1}{x_i}\,dt+k_2\sum_{j=1:j\neq i}^N \Big\{\frac{1}{x_i-x_j}+\frac{1}{x_i+x_j}\Big\}\,dt\\
        &=dB_i+\frac{k_1}{x_i}\,dt+k_2\sum_{j=1:j\neq i}^N \frac{2x_i}{x_i^2-x_j^2}\,dt.
    \end{align*}
    By Corollary~\ref{cor:multiBessel}, the corresponding Hausdorff dimension of collision times is
    \[\dim x^{-1}(\Cb_B)=\max\Big\{0,\frac12-\min\{k_1,k_2\}\Big\}.\]
    The parameters $k_1$ and $k_2$ control the type of collision that occurs. Indeed, by \cite[Prop.~1]{Demni09}, the first hitting time of zero is finite if $k_1<1/2$, and infinite otherwise, while the first collision time between particles is finite if $k_2<1/2$ and infinite otherwise. This observation will be useful for the next example.
    
    \subsection{Wishart processes}
    
    The Wishart process \cite{Bru91,KonigOConnell01}  is another important system that can be formulated as the matrix eigenvalue process of the product of a rectangular matrix with independent Brownian motion entries with its transpose in the real case, or its conjugate transpose in the complex case. As in the Dyson case, the values $\kappa=1$ and $2$ correspond to real and complex matrix formulations of the system, but here we consider the general case given by the SDE
    \begin{align*}
        dy_i&=2\sqrt{y_i}\,dB_i+\kappa a\,dt+\kappa\sum_{j=1:j\neq i}^N\frac{y_i+y_j}{y_i-y_j}\,dt.
    \end{align*}
    Here, $\kappa$ and $a$ are positive parameters, and particles are in the Weyl chamber $\Cc_{B_N}=\{x\in\R^N:0\leq x_i\leq x_j, 1\leq i<j\leq N\}$, but the repulsive interaction between particles is given by the root system $A_{N-1}$.
    The Wishart process can be obtained from the multivariate Bessel process of type $B_N$ by the relationships $y_i=x_i^2$, $\kappa a=2k_1+2k_2(N-1)+1$, and $\kappa=2k_2$.
    This implies that Wishart process particles do not hit zero provided
    \[k_1\geq1/2,\ \text{equivalently}\ a\geq\frac2\kappa+N-1.\]
    Under these conditions, the type-$B_N$ process does not hit the origin at positive times. Since its coordinates are nonnegative, the transformation $y_i=x_i^2$ preserves pair-collision times:
    \[
        y_i(t)=y_j(t)\quad\Longleftrightarrow\quad x_i(t)=x_j(t).
    \]
    The constant-multiplicity dimension formula, with $k_2=\kappa/2$ and $k_1\geq1/2$, therefore gives the following corollary.
    \begin{corollary}
        When $a\geq 2/\kappa +N-1$, the Hausdorff dimension of collision times between particles in a Wishart process is given by
        \[\dim y^{-1}(\Cb_B)=\max\Big\{0,\frac{1-\kappa}{2}\Big\}.\]
    \end{corollary}
    
    \subsection{Jacobi processes}
    Jacobi processes are mutually repelling particle systems in a finite segment of the real line. As introduced in \cite{Demni10}, all particles are ordered and lie in the interval $(0,1)$, but for the sake of symmetry we consider the process on $(-1,1)$, namely, the $N$ particles $\{\lambda_i\}_{i=1}^N$ in \cite{Demni10} are replaced by $x_i=2\lambda_i-1$, and the corresponding SDE reads
    \begin{align*}
        dx_i=2\sqrt{1-x_i^2}\,dB_i+k\bigg[p-q-(p+q)x_i+2\sum_{j=1:j\neq i}^N\frac{1-x_ix_j}{x_i-x_j}\bigg]\,dt,
    \end{align*}
    where we have replaced the parameter $\beta>0$ by $k$ for notational convenience, and $p$ and $q$ are integer parameters satisfying $\min\{p,q\}\geq N-1+2/k$ in order to ensure the SDE has a unique strong solution.
    We note here that the cases $k=1$ and $2$ can be formulated as dynamical extensions of the Jacobi ensembles, very much in the same way as the Wishart and Dyson cases above.
    From the SDE, it is clear that $k>0$ governs the repulsion between particles, while $p$ and $q$ represent the repulsion strength of the left and right walls respectively.

    Similar to the Wishart process, we see that $\sigma(u)=2\sqrt{1-u^2}$, $b(u)=k[p-q-(p+q)u]$, and $k_{i,j}(x)=2k(1-x_ix_j)$. 
    Because the condition $\min\{p,q\}\geq N-1+2/k$ ensures that $k[p-(N-1)]\geq2$ and $k[q-(N-1)]\geq2$, $x_1$ will not hit $-1$, and $x_N$ will not hit $1$ almost surely, so effectively \ref{ass:sigma:k:positive}-\ref{ass:k:monotonicty} are met here as well.
    Indeed, both $\sigma(x_i)$ and $k_{i,j}(x)$ are effectively positive as no particle hits $\pm1$, $b(x_i)$ is a decreasing function, and for $m\leq i<j\leq n$,
    \[\frac{1-x_ix_j}{x_j-x_i}-\frac{1-x_mx_n}{x_n-x_m}=\frac{(x_n-x_j)(1-x_ix_m)+(x_i-x_m)(1-x_jx_n)}{(x_j-x_i)(x_n-x_m)}\geq0,\]
    so \ref{ass:k:monotonicty} holds as
    \begin{align*}
        \frac{k_{i,j}(x)}{x_j-x_i}=2k\frac{1-x_ix_j}{x_j-x_i}\geq 2k\frac{1-x_mx_n}{x_n-x_m}=\frac{k_{m,n}(x)}{x_n-x_m}.
    \end{align*}
    Furthermore,
    \begin{align}
        \frac{|\alpha|^2k_\alpha(x)}{\sum_{r=1}^N\alpha_r^2\sigma^2(x_r)}
        &=\frac{k(1-x_ix_j)}{2-x_i^2-x_j^2}=\frac{k}{2}\left(1+
        \frac{(x_j-x_i)^2}{2-x_i^2-x_j^2}\right)
        \geq\frac{k}{2}.
        \label{eq:jacobi:k:sigma}
    \end{align}
    which implies that
    \begin{align*}
        \dim x^{-1}(\Cb)\leq\max\Big\{0,\frac{1-k}2\Big\}.
    \end{align*}
    Note that we do not have a non-trivial lower bound here. The needed bounds hold after stopping on compact subsets of the coordinate domain $(-1,1)^N$. More precisely, put
    \[
        \tau_m:=\inf\left\{t\geq0:
             \min_i(1-|x_i(t)|)\leq\frac1m\right\}\wedge m.
    \]
    Since the endpoints are not hit, continuity gives $\tau_m\uparrow\infty$ almost surely. On these compact localizations the coefficients satisfy the positivity and compatibility conditions, and the upper-bound argument applies with finite local constants. Countable stability gives the stated global upper bound.
    The unbounded global supremum of this coefficient ratio makes the present supremum-based lower bound noninformative in this example. This does not rule out almost-sure collisions. For $0<k<1$, finiteness of the first particle-collision time is established in \cite[Proposition~3.1]{Demni10}.

\section*{Acknowledgments}
    SA was supported by the Kakenhi grant number JP19K14617 of the Japanese Society for the Promotion of Science (JSPS). We are grateful to the referees for helpful comments that have led to an improvement of the paper's presentation.
\bibliography{bibtex}

@Book{Adler1981,
	title     = {The Geometry of Random Fields},
	publisher = {Society for Industrial and Applied Mathematics (SIAM, 3600 Market Street, Floor 6, Philadelphia, PA 19104)},
	year      = {1981},
	author    = {Adler, R.J.},
	series    = {Classics in Applied Mathematics},
	isbn      = {9780898718980},
	url       = {https://books.google.de/books?id=8a\_-VEL7GUYC},
}

@Article{LiuXiao98,
	author  = {Liu, Luqin and Xiao, Yimin},
	title   = {Hausdorff dimension theorem for self-similar {M}arkov processes},
	journal = {Probability and Mathematical Statistics},
	year    = {1998},
	volume  = {18},
	number  = {2},
	pages   = {369--383},
}

@Article{RoslerVoit98,
	author  = {R{\"o}sler, Margit and Voit, Michael},
	title   = {Markov Processes Related with {D}unkl Operators},
	journal = {Advances in Applied Mathematics},
	year    = {1998},
	volume  = {22},
	pages = {575--643},
}

@Article{Demni09,
	author  = {Demni, Nizar},
	title   = {Radial {D}unkl processes: {E}xistence, uniqueness and hitting time},
	journal = {Comptes Rendus de l'Acad{\'e}mie des Sciences de Paris, S{\'e}rie I},
	year    = {2009},
	volume  = {347},
	pages = {1125--1128},
}

@Book{Falconer2013fractal,
	title     = {Fractal Geometry: Mathematical Foundations and Applications},
	publisher = {Wiley},
	year      = {2013},
	author    = {Falconer, K.},
	isbn      = {9781118762868},
	lccn      = {2013031708},
	url       = {https://books.google.de/books?id=XJN7AgAAQBAJ},
}

@Book{Forrester10,
	title     = {Log-gases and Random Matrices},
	publisher = {Princeton University Press},
	year      = {2010},
	author    = {Forrester, P.J.},
	series    = {London Mathematical Society Monographs},
	isbn      = {9780691128290},
}

@Article{Dyson62A,
	author        = {Dyson, F.J.},
	title         = {A {B}rownian-Motion Model for the Eigenvalues of a Random Matrix},
	journal       = {J. Math. Phys.},
	year          = {1962},
	volume        = {3},
	number        = {6},
	pages         = {1191--1198},
	url           = {https://doi.org/10.1063/1.1703862},
}

@Article{KatoriTanemura04,
	author        = {Katori, M. and Tanemura, H.},
	title         = {Symmetry of matrix-valued stochastic processes and noncolliding diffusion particle systems},
	journal       = {J. Math. Phys.},
	year          = {2004},
	volume        = {45},
	pages         = {3058--3085},
	url           = {https://doi.org/10.1063/1.1765215},
}

@Article{Bru91,
	author        = {Bru, M.-F.},
	title         = {Wishart Processes},
	journal       = {J. Theor. Probab.},
	year          = {1991},
	volume        = {4},
	number        = {4},
	pages         = {725--751},
	url           = {https://doi.org/10.1007/BF01259552},
}

@Article{KonigOConnell01,
	author        = {K{\"o}nig, W. and O'Connell, N.},
	title         = {Eigenvalues of the {L}aguerre process as non-colliding squared {B}essel processes},
	journal       = {Electron. Comm. Probab.},
	year          = {2001},
	volume        = {6},
	pages         = {107--114},
	url           = {https://doi.org/10.1214/ECP.v6-1040},
}

@Article{Demni10,
	author        = {Demni, N.},
	title         = {$\beta$-{J}acobi processes},
	journal       = {Adv. Pure Appl. Math.},
	year          = {2010},
	volume        = {1},
	pages         = {325--344},
	url           = {https://doi.org/10.1515/APAM.2010.019},
}

@Article{CepaLepingle,
	author        = {C{\'e}pa, E. and L{\'e}pingle, D.},
	title         = {Diffusing particles with electrostatic repulsion},
	journal       = {Probab. Theory Relat. Fields},
	year          = {1997},
	volume        = {107},
	pages         = {429--449},
}

@Article{GraczykMalecki,
	author = {Graczyk, P. and Ma\l{}ecki, J.},
	title         = {Strong solutions of non-colliding particle systems},
	journal       = {Electron. J. Probab.},
	year          = {2014},
	volume        = {19},
	number        = {119},
	pages         = {1--21},
	url           = {http://dx.doi.org/10.1214/EJP.v19-3842},
}

@Article{GraczykMalecki13,
	author        = {Graczyk, P. and Ma\l{}ecki, J.},
	title         = {Multidimensional {Y}amada-{W}atanabe theorem and its
applications to particle systems},
	journal       = {J. Math. Phys.},
	year          = {2013},
	volume        = {54},
	pages         = {021503},
}

@Article{GraczykMalecki19,
	author        = {Graczyk, P. and Ma\l{}ecki, J.},
	title         = {On squared {B}essel particle systems},
	journal       = {Bernoulli},
	year          = {2019},
	volume        = {25},
	number        = {2},
	pages         = {828--847},
	url           = {https://doi.org/10.3150/17-BEJ997},
}

@Article{GallardoYor05,
	author        = {Gallardo, L. and Yor, M.},
	title         = {Some new examples of {M}arkov processes which enjoy the time-inversion property},
	journal       = {Probab. Theory Relat. Fields},
	year          = {2005},
	volume        = {132},
	pages         = {150--162},
	url           = {https://doi.org/10.1007/s00440-004-0399-y},
}

@Article{GoeingJaeshckeYor03,
	author        = {G{ö}ing-Jaeschke, A. and Yor, M.},
	title         = {A survey and some generalizations of {B}essel processes},
	journal       = {Bernoulli},
	year          = {2003},
	volume        = {2},
	pages         = {313--349},
}

@Article{HufnagelAndraus24,
	author        = {Hufnagel, N. and Andraus, S.},
	title         = {Hausdorff dimension of collision times in one-dimensional log-gases},
	journal       = {AIP Advances},
	year          = {2024},
	volume        = {14},
	pages         = {1--16,{ paper no. }095123},
	url           = {https://doi.org/10.1063/5.0148019},
}

@Book{Humphreys,
  title     = {Reflection Groups and Coxeter Groups},
  publisher = {Cambridge University Press},
  year      = {1997},
  author    = {Humphreys, J. E.},
  isbn      = {9780521436137},
}

@Book{McKean,
  title     = {Stochastic Integrals},
  publisher = {Academic Press},
  year      = {1969},
  author    = {McKean, H. P.}
}

@unpublished{Malecki25,
	author        = {Ma\l{}ecki, J.},
	title         = {Generalized {B}essel-{D}unkl diffusions},
	journal       = {preprint},
	year          = {2026},
    note = {https://arxiv.org/abs/2605.23529}
}

@unpublished{Yuan24,
    author = {Yuan, W.},
    title = {Multiple collisions of eigenvalues and singular values of matrix {G}aussian field},
    note = {https://arxiv.org/abs/2407.09070}, 
    year = 2024
}

@InCollection{CepaLepingle07,
  author    = {C{\'e}pa, E. and L{\'e}pingle, D.},
  title     = {No multiple collisions for mutually repelling {B}rownian particles},
  booktitle = {S{\'e}minaire de Probabilit{\'e}s XL},
  series    = {Lecture Notes in Mathematics},
  volume    = {1899},
  pages     = {241--246},
  publisher = {Springer},
  year      = {2007},
  url       = {https://doi.org/10.1007/978-3-540-71189-6_12},
}
\end{document}